\documentclass[centertags,reqno,draft]{amsart}
\usepackage{amssymb, latexsym}
\usepackage{verbatim, enumerate}

\DeclareMathOperator{\im}{Im}
\DeclareMathOperator{\re}{Re}
\renewcommand{\Re}{\text{\rm Re}}
\renewcommand{\Im}{\text{\rm Im}}
\newcommand{\f}{\frac}
\newcommand{\bi}{\bibitem}
\newcommand{\no}{\notag}
\newcommand{\lb}{\label}
\newcommand{\ol}{\overline}

\newcommand{\wti}{\widetilde}

\newcommand{\eps}{{\varepsilon}}
\newcommand{\al}{\alpha}
\newcommand{\be}{\beta}
\newcommand{\de}{\delta}
\newcommand{\ga}{\gamma}
\newcommand{\te}{\theta}
\newcommand{\C}{{\mathbb C}}
\newcommand{\D}{{\mathbb D}}

\newcommand{\N}{{\mathbb N}}
\newcommand{\R}{{\mathbb R}}
\newcommand{\Z}{{\mathbb Z}}

\newcommand{\cL}{{\mathcal L}}
\newcommand{\cM}{{\mathcal M}}

\newcommand{\cK}{{\mathcal K}}
\newcommand{\dD}{{\partial\D}}
\newcommand{\EC}{{\mathcal E}}
\newcommand{\cmv}{{\mathcal C}}
\newcommand{\ac}{\rm{ac}}

\theoremstyle{plain}
\newtheorem{theorem}{Theorem}
\newtheorem{lemma}[theorem]{Lemma}

\newtheorem{proposition}[theorem]{Proposition}

\newtheorem{corollary}[theorem]{Corollary}

\theoremstyle{definition}
\newtheorem{definition}[theorem]{Definition}
\newtheorem{example}[theorem]{Example}
\theoremstyle{remark}
\newtheorem*{remark}{Remark}

\numberwithin{equation}{section}
\numberwithin{theorem}{section}

\begin{document}

\title[Right limits for CMV matrices]{Right limits and reflectionless measures\\ for CMV matrices}
\author{Jonathan Breuer, Eric Ryckman, and Maxim Zinchenko}
\address{Mathematics 253-37, California Institute of Technology, Pasadena CA 91125-0001, USA}
\email{jbreuer@caltech.edu}
\email{eryckman@caltech.edu}
\email{maxim@caltech.edu}
\date{\today}

\begin{abstract}
We study CMV matrices by focusing on their right-limit sets. We
prove a CMV version of a recent result of Remling dealing with the
implications of the existence of absolutely continuous spectrum, and
we study some of its consequences. We further demonstrate the
usefulness of right limits in the study of weak asymptotic
convergence of spectral measures and ratio asymptotics for
orthogonal polynomials by extending and refining earlier results of
Khrushchev. To demonstrate the analogy with the Jacobi case, we
recover corresponding previous results of Simon using the same
approach.

\end{abstract}

\maketitle

%
%
%
%

\section{Introduction}
This paper considers some issues in the spectral theory of CMV
matrices viewed through the lens of the notion of right limits. In
particular, a central theme will be the fact that one may use the
properties of right limits of a given CMV matrix to deduce relations
between the asymptotics of its entries and its spectral measure.

CMV matrices (see Definition \ref{CMV-matrix} below) were named
after Cantero, Moral and Vel\'azquez \cite{CMV} and may be described
as the unitary analog of Jacobi matrices: they arise naturally in
the theory of orthogonal polynomials on the unit circle (OPUC) in
much the same way that Jacobi matrices arise in the theory of
orthogonal polynomials on the real line (OPRL).

Two related topics will be at the focus of our discussion.  The first is the extension to the CMV setting of
a collection of results, proven recently by Remling \cite{Remling},
describing various consequences of the existence of absolutely
continuous spectrum of Jacobi matrices. The second topic is the
simplification of various elements of Khrushchev's theory of weak
limits of spectral measures, through the understanding that the
matrices at the center of attention have right limits in a very
special class.

As we shall see, these two subjects are intimately connected through
the notion of reflectionless whole-line CMV matrices. This is a
concept that has been extensively investigated in recent years, in
the context of both CMV and Jacobi matrices (\cite{Cr89},
\cite{DS83}, \cite{GKT96}--\cite{GY06}, \cite{GZ09},
\cite{Jo82}--\cite{KK88}, \cite{MPV08}--\cite{PR08}, \cite{RemSc},
\cite{Remling}, \cite{Si07a}--\cite{SY97}) and was seen to have
numerous applications in their spectral theory. There are various
definitions of this notion, all of which turn out to be equivalent
in the Jacobi matrix case. We shall show that this is not true in
the CMV case. In particular, we construct an example of a whole-line
CMV matrix that is not reflectionless in the spectral-theoretic
sense, while all of its diagonal spectral measures are
reflectionless in the measure-theoretic sense. We will show,
however, that this may only happen for a very limited class of CMV
matrices. Their existence in the CMV case, together with Remling's
Theorem (Theorem~\ref{Remling} below), provides for a simple proof
of Khrushchev's Theorem (Theorem~\ref{Khrushchev} below).

We should remark that ours is not the first paper to deal with right limits of CMV matrices.  For other examples and related results, see for instance \cite{Gol-Nevai} and \cite{Last-Simon}.

In order to describe our results, some notation is needed: given a
probability measure, $\mu$, on the boundary of the unit disc,
$\partial \D$, we let $\{ \Phi_n(z) \}_{n=0}^\infty$ and $\{
\varphi_n(z) \}_{n=0}^\infty$ denote the monic orthogonal
 and the orthonormal polynomials one gets by applying
the Gram--Schmidt procedure to $1,z,z^2,\dots$ (we assume throughout
that the support of $\mu$ is an infinite set so the polynomial
sequences are indeed infinite). The $\Phi_n$ satisfy the Szeg\H{o}
recurrence equation:
\begin{equation} \lb{SzRec}
\Phi_{n+1}(z)=z\Phi_n(z)-\overline{\alpha_n} \Phi_n^*(z)\end{equation} where
$\{\alpha_n\}_{n=0}^\infty$ is a sequence of parameters satisfying
$|\alpha_n|<1$ and $\Phi_n^*(z)=z^n \overline{\Phi_n(1/\overline{z})}$. We call
$\{\alpha_n\}_{n=0}^\infty$ the Verblunsky coefficients associated with $\mu$.
As is well known \cite{OPUC1}, the sequence $\{\alpha_n\}_{n=0}^\infty$ may be
used to construct a semi-infinite 5-diagonal matrix, $\cmv$, (called the CMV
matrix) such that the operator of multiplication by $z$ on $L^2(\partial \D,
d\mu)$ is unitarily equivalent to the operator $\cmv$ on $\ell^2(\Z_+)$
($\Z_+=\{0,1,2,\ldots\}$). Explicitly, $\cmv$ is given by
\begin{equation} \label{CMV-matrix}
\cmv = \begin{pmatrix} {}& \bar\alpha_0 & \bar\alpha_1 \rho_0 &
\rho_1 \rho_0 & 0 & 0 & \dots &
{} \\
{}& \rho_0 & -\bar\alpha_1 \alpha_0 & -\rho_1 \alpha_0 & 0 & 0 &
\dots &
{} \\
{}& 0 & \bar\alpha_2\rho_1 & -\bar\alpha_2 \alpha_1 & \bar\alpha_3
\rho_2 & \rho_3 \rho_2 & \dots & {} \\
{}& 0 & \rho_2 \rho_1 & -\rho_2 \alpha_1 & -\bar\alpha_3 \alpha_2 &
-\rho_3 \alpha_2 & \dots & {} \\
{}& 0 & 0 & 0 & \bar\alpha_4 \rho_3 & -\bar\alpha_4 \alpha_3 & \dots
&
{} \\
{}& \dots & \dots & \dots & \dots & \dots & \dots & {}
\end{pmatrix}
\end{equation}
with $\rho_n=\left( 1-|\alpha_n|^2 \right)^{1/2}$.

Now, for a probability measure $\mu$ on $\partial \D$, let
$d\mu(\theta)=w(\theta)d\theta+d\mu_{\textrm{sing}}(\theta)$ be the
decomposition into its absolutely continuous and singular parts (with respect
to the Lebesgue measure). If $\cmv$ is the corresponding CMV matrix, we define
the \emph{essential support of the absolutely continuous spectrum of $\cmv$} to
be the set $\Sigma_{\ac}(\cmv) \equiv \{\theta \mid w(\theta)>0 \}$. Clearly,
$\Sigma_{\ac}(\cmv)$ is only defined up to sets of Lebesgue measure zero, so
the symbol and the name should be understood as representing elements in an
equivalence class of sets rather than a particular set. For the sake of
simplicity we will ignore this point in our discussion.

The first part of this paper deals with proving the analog of Remling's Theorem
(Theorem 1.4 in \cite{Remling}) for CMV matrices and deriving some
consequences. In a nutshell, Remling's Theorem for CMV matrices says that for
any given CMV matrix, $\cmv$, all of its right limits are reflectionless on
$\Sigma_{\ac}(\cmv)$ (see Definitions \ref{rlimit def} and \ref{reflectionless} and Theorem~\ref{Remling}
below). Here is a consequence that will also provide the link to Khrushchev's
theory (the Jacobi analog was stated and proved in \cite{Remling}):

\begin{theorem} \label{sparse}
Let $\{\alpha_n\}_{n=0}^\infty$ and
$\{\widetilde{\alpha}_n\}_{n=0}^\infty$ be two sequences of
Verblunsky coefficients such that $($with $\Delta_n = \widetilde{\alpha}_n - \alpha_n$$)$\\
$(i)$ $|\alpha_n|<1$, $|\widetilde{\alpha}_n|<1$ for all $n$. \\
$(ii)$ There exist sequences $\{m_j\}_{j=1}^\infty$,
$\{n_j\}_{j=1}^\infty$ with $n_j-m_j \rightarrow \infty$ so that
$$\lim_{j \rightarrow \infty} \sup_{m_j \leq n
<n_j}|\Delta_n|=0.$$
$(iii)$ $\limsup_{j \rightarrow \infty}|\Delta_{n_j}|>0$. \\
Furthermore, let $\cmv$ and $\widetilde{\cmv}$ denote the CMV matrices
of $\{\alpha_n\}_{n=0}^\infty$ and
$\{\widetilde{\alpha}_n\}_{n=0}^\infty$ respectively. Then
\begin{equation} \label{sparse-ac} \rm{Leb}\left(\Sigma_{\ac} \bigl( \cmv \bigr) \cap \Sigma_{\ac} \bigl( \widetilde{\cmv} \bigr) \right)=0 \end{equation}
where $\rm{Leb}(\cdot)$ denotes Lebesgue measure. In particular, if
$\cmv$ is associated with a sequence of Verblunsky coefficients
satisfying
\begin{equation} \label{mate-nevai-sparse}
\forall k \geq 1 \ \lim_{n \rightarrow \infty}\alpha_n \alpha_{n+k}=0, \quad
\quad \limsup_{n \rightarrow \infty} |\alpha_n|>0,
\end{equation}
then $\cmv$ has purely singular spectrum.
\end{theorem}

In order to state Remling's Theorem we need some more terminology.

\begin{definition}\label{rlimit def}
Given a sequence of Verblunsky coefficients
$\{\alpha_n\}_{n=0}^\infty$, a doubly-infinite sequence of
parameters $\{ \widetilde{\alpha}_n \}_{n \in \Z}$ with
$|\widetilde{\alpha}_n| \leq 1$ is called a \emph{right limit} of
$\{\alpha_n\}_{n=0}^\infty$ if there is a sequence of integers $n_j \rightarrow
\infty$ such that $\forall n \in \Z$, $\widetilde{\alpha}_n=\lim_{j
\rightarrow \infty} \alpha_{n+n_j}$.
\end{definition}

Since a sequence of Verblunsky coefficients is always bounded, by
compactness it always has at least one (and perhaps many) right
limits. Given a doubly infinite sequence $\{ \widetilde{\alpha}_n
\}_{n \in \Z}$, one may also define a corresponding unitary matrix
on $\ell^2(\Z)$, extending the half-line matrices to the left and
top (see \eqref{whole-line-cmv} for the precise form). We call such
a matrix the corresponding whole-line CMV matrix and denote it by
$\EC$. For this reason, we shall often refer to a doubly infinite
sequence of numbers $\{ \widetilde{\alpha}_n \}_{n \in \Z}$ with
$|\widetilde{\alpha}_n| \leq 1$ as a (doubly infinite) sequence of
Verblunsky coefficients. If $\{ \widetilde{\alpha}_n \}_{n \in \Z}$
is a right limit of $\{\alpha_n\}_{n=0}^\infty$, we refer to the
corresponding whole-line CMV matrix as a right limit of the
half-line CMV matrix associated with $\{\alpha_n\}_{n=0}^\infty$.

Recall that any probability measure $\mu$ on $\partial \D$ may be
naturally associated with a Schur function $f$ (an analytic function
on $\D$ satisfying $\sup_{z \in \D}|f(z)| \leq 1$) and a
Carath\'eodory function $F$ (an analytic function on $\D$ satisfying
$F(0)=1$ and $\re F(z)>0$ on $\D$). This is given by
\begin{equation} \label{carat-schur} \frac{1+zf(z)}{1-zf(z)} =
F(z)=\int_0^{2\pi} \frac{e^{i\theta}+z}{e^{i\theta}-z}\,
d\mu(\theta).\end{equation} The correspondence is 1-1 and onto. By a
classical result, $\lim_{r \uparrow 1}F(r e^{i\theta})$ and $\lim_{r
\uparrow 1}f(r e^{i\theta})$ exist Lebesgue a.e.\ on $\partial \D$.
We denote them by $F(e^{i\theta})$ and $f(e^{i\theta})$ respectively
and, when there is no danger of confusion, simply by $F(z)$ or
$f(z)$ for $z \in \partial \D$.

Given a Schur function, $f$, let $f_0 = f$ and define a sequence of
Schur functions $f_n$ and parameters $\gamma_n \in \D$ by $$z
f_{n+1}(z) = \frac{f_n(z) - \gamma_n}{1 - \overline \gamma_n
f_n(z)}, \quad\quad\quad \gamma_n = f_n(0).$$ If for some $n$,
$|\gamma_n|=1$, we stop and the Schur function is a finite Blaschke
product. Otherwise, we continue. It is known \cite{OPUC1} that this
process (known as the Schur algorithm) sets up a 1-1 correspondence
between Schur functions $f$ and parameter sequences $\gamma_n \in
\D$, so given any sequence $\gamma_n \in \D$ there is a unique Schur
function $f(z;\gamma_0, \gamma_1,\dots)$ associated to it in the
above way. The $\gamma$'s are frequently termed the \emph{Schur
parameters} associated to $f$ (or equivalently, to $\mu$ or $F$).
Geronimus's Theorem \cite{Geronimus} says that $\gamma_n=\alpha_n$
(the Verblunsky coefficients of $\mu$ appearing above). Finally,
note that by definition
\begin{equation}\label{sa}
f_n(z; \alpha_0,\alpha_1,\dots) = f(z; \alpha_n,\alpha_{n+1},\dots).
\end{equation}

For a doubly infinite sequence of Verblunsky coefficients,
$\{\alpha_n\}_{n \in \Z}$ (some of which may lie on $\partial \D$),
we define two sequences of Schur functions:
\begin{equation}\label{f}
f_+(z,n) = f(z;\alpha_n , \alpha_{n+1}, \dots) \quad\text{and}\quad
f_-(z,n) = f(z; -\overline{\alpha_{n-1}} , -\overline{\alpha_{n-2}}
, \dots)
\end{equation}
where as usual, if one of the $\alpha$'s lies in $\partial \D$ then
we stop the Schur algorithm at that point and the corresponding
Schur function is a finite Blaschke product.

\begin{definition}\label{reflectionless}
Let $\{\alpha_n\}_{n \in \Z}$ be a doubly-infinite sequence of Verblunsky
coefficients and let $\EC$ be the associated whole-line CMV matrix. Given a
Borel set $A \subseteq
\partial \D$, we will say $\EC$ is \emph{reflectionless} on $A$ if for all
$n \in \Z$, $$zf_+(z,n) = \overline{f_-(z,n)}$$ for Lebesgue
almost every $z \in A$. By the Schur algorithm, one can easily see
that ``for all $n \in \Z$'' may be replaced with ``for some $n \in
\Z$.''
\end{definition}

\begin{remark}
The analogous definition for whole-line Jacobi matrices involves a
similar relationship between the left and right $m$-functions.
\end{remark}

The following is the CMV version of Remling's Theorem:

\begin{theorem}[Remling's Theorem for CMV matrices] \label{Remling}
Let $\cmv$ be a half-line CMV matrix, and let $\Sigma_{\ac}(\cmv)$
be the essential support of the absolutely continuous part of the
spectral measure. Then every right limit of $\cmv$ is reflectionless
on $\Sigma_{\ac}(\cmv)$.
\end{theorem}

Remling's proof in the Jacobi case relies on previous work by
Breimesser and Pearson \cite{BP1, BP2} concerning convergence of
boundary values for Herglotz functions. We will prove
Theorem~\ref{Remling} using the analogous theory for Schur
functions. For a CMV matrix, $\cmv$, recall that its essential
spectrum,  $\sigma_{\rm{ess}}\left(\cmv \right)$, is its spectrum
with the isolated points removed. The following extension of a
celebrated theorem of Rakhmanov is a simple corollary of
Theorem~\ref{Remling}:

\begin{theorem} \label{Rakhmanov}
Assume $\Sigma_{\rm{ac}} \left(\cmv \right)=\sigma_{\rm{ess}}
\left(\cmv \right)=A$ where $\sigma_{\rm{ess}} \left(\cmv \right)$
is the essential spectrum of $\cmv$. Then for any right limit $\EC$
of $\cmv$, $\sigma(\EC) = A$ and  $\EC$ is reflectionless on $A$.
\end{theorem}

\begin{remark}
In the case that $A$ is a finite union of intervals, the
corresponding class of whole-line CMV matrices is called the
\emph{isospectral torus} of $A$, since it has a natural torus
structure \cite{GZ06b,PR}. If $A=\partial \D$, the isospectral torus
is known to consist of a single point---the CMV matrix with
Verblunsky coefficients all equal to zero \cite{GZ06b}. Thus, one
gets Rakhmanov's Theorem \cite{Rakhmanov1,Rakhmanov2} as a
corollary.
\end{remark}

\begin{proof}[Proof of Theorem \ref{Rakhmanov}]
Let $\EC$ be a right limit of $\cmv$ and $\{\delta_n\}_{n \in \Z}$
be the standard orthonormal bais for $\ell^2(\Z)$. For $\psi=\sum_{n
\in \Z} 2^{-|n|} \delta_n$, let
$d\mu_\psi(\theta)=w_{\psi}(\theta)d\theta+d\mu_{\psi,\textrm{sing}}$
be the spectral measure of $\psi$ and $\EC$. Let $\Sigma_{\ac}\left(
\EC \right)=\{\theta \mid w_\psi(\theta)>0 \}$ (defined, again, up
to sets of Lebesgue measure zero).

By Theorem~\ref{Remling}, $A \subseteq \Sigma_{\ac}\left( \EC
\right)$ (up to a set of Lebesgue measure zero), since the
reflectionless condition implies positivity of the real part of the
Carath\'eodory function associated with $d\mu_\psi$. Also,
$\sigma\left( \EC \right) \subseteq \sigma_{\textrm{ess}}\left( \cmv
\right)=A$ by approximate-eigenvector arguments (see for instance
\cite{Last-Simon}). Since obviously $\Sigma_{\ac}\left( \EC \right)
\subseteq \sigma(\EC)$,  we have equality throughout. The
reflectionless condition now follows from Theorem \ref{Remling}.
\end{proof}

\begin{remark}
Using Theorem \ref{Remling} and a bit of work, one can also derive
parts of Kotani theory for ergodic CMV matrices (see for instance
\cite[Sect.\! 10.11]{OPUC2}).  Remling also obtains deterministic versions of these results for Jacobi matrices (see \cite[Thm's 1.1 and 1.2]{Remling}).  His proofs extend directly to the CMV case we are considering, so we will not pursue this here.
\end{remark}

Corresponding to the notion of reflectionless operators, there is
also the notion of reflectionless \emph{measures}:

\begin{definition} \label{measure-reflectionless}
A probability measure $\mu$ on $\partial \D$ is said to be
\emph{reflectionless} on a Borel set $A \subseteq \partial \D$ if
the corresponding Carath\'eodory function $F$ has $\im
F(e^{i\theta})=0$ for Lebesgue a.e.\ $e^{i\theta}\in A$.
\end{definition}

\begin{remark}
The analogous definition for measures on the real line involves the
vanishing of the real part of the Borel (a.k.a.\ Cauchy or
Stieltjes) transform of $\mu$ (see for instance \cite{Teschl}).
\end{remark}
\begin{remark}
There is also a natural dynamical notion for when an operator is
reflectionless. For the relationship between this and spectral
theory see \cite{BRS}.
\end{remark}

Reflectionless Jacobi matrices and reflectionless measures on $\R$
are related in the following way: given a whole-line Jacobi matrix,
$H$, let $\mu_n$ be the spectral measure of $H$ and $\delta_n$
($\delta_n \in \ell^2(\Z)$ is defined by $\delta_n(j)=\delta_{n,j}$
with $\delta_{n,j}$ the Kronecker delta). Then $H$ is reflectionless
on $A \subseteq \R$ if and only if $\mu_n$ are reflectionless on $A$
for all $n \in \Z$ (again, see \cite{Teschl}). A fact we would like
to emphasize in this paper is that the analogous statement
\emph{does not} hold for CMV matrices.

\begin{example} \label{example}
Fix $j_0 \in \Z$ and some $0 < |\beta| <1$, and let $\{\alpha_n\}_{n \in \Z}$
be the sequence of Verblunsky coefficients defined by
$$
\alpha_n = \begin{cases}
\beta &  n=j_0\\
0 & \textrm{otherwise.}
\end{cases}
$$
Let $\EC$ be the CMV matrix for these $\alpha$'s.  From the Schur algorithm we see
\begin{equation}\label{sfs}
f(z;0,0,\dots) = 0 \quad\text{and}\quad f(z;\beta,0,0,\dots) = \beta,
\end{equation}
so $\EC$ is \emph{not} reflectionless anywhere.

On the other hand, let $\mu_n$ be the spectral measure of $\EC$ and
$\delta_n$, and let $f(z,n)$ its corresponding Schur function.  It
is shown in \cite{GZ06a} (see also \cite{Khrushchev0} for the
analogous formula in the half-line case) that
\begin{equation} \lb{diag-schur}
f(z,n)=f_+(z,n)f_-(z,n), \quad z\in\D,\; n\in\Z.
\end{equation}
Thus, for any $n \in \Z$, \eqref{sfs} implies $d\mu_n(\theta)=\frac{d \theta}{2
\pi}$.  In particular, $\mu_n$ is reflectionless on all of $\partial \D$ while $\EC$ is
not reflectionless on any subset of positive Lebesgue measure.
\end{example}

We will show, however, that this is the only example of such
behavior:

\begin{theorem} \label{equivalence}
Let $\EC$ be the whole-line CMV matrix corresponding to the sequence
$\{\alpha_n \}_{n \in \Z}$,  satisfying $\alpha_n \neq 0$ for at least two
different $n \in \Z$. Then $\EC$ is reflectionless on $A \subseteq
\partial \D$ if and only if $\mu_n$ is reflectionless on $A$ for all
$n$.
\end{theorem}

The connection between the above result and Khrushchev's theory of weak
limits comes from the fact that, together with Example
\ref{example}, Theorem~\ref{sparse} provides for a particularly
simple proof of the following theorem of Khrushchev.

\begin{theorem}[Khrushchev \cite{Khrushchev0}]\label{Khrushchev}
Let $\cmv$ be a CMV matrix with Verblunsky coefficients
$\{\alpha_n\}_{n=0}^\infty$ and measure $\mu$, and let
$d\mu_n(\theta)=|\varphi_{n}(e^{i\theta})|^2d\mu(\te)$. Then
$$d\mu_n(\theta) \rightarrow \frac{d\theta}{2 \pi}$$ weakly if and only if
$$\forall k \geq 1 \ \lim_{n \rightarrow \infty}\alpha_n
\alpha_{n+k}=0.$$ Furthermore, these conditions imply
that either $\alpha_n \rightarrow 0$ or $\mu$ is purely singular.
\end{theorem}

This theorem is naturally a part of a larger discussion dealing with
weak limits of $\mu_n$. In particular, Khrushchev's theory deals
with the cases in which such weak limits exist. We will show that
the analysis of these cases becomes simple when performed using
right limits. The reason for this is that the $\mu_n$ above are
actually the spectral measures of $\cmv$ and $\delta_n$ and, along a
proper subsequence, these converge weakly to the corresponding
spectral measures of the right limit. Thus, if $\mu_n$ converges
weakly to $\nu$ as $n \rightarrow \infty$, all of the diagonal
measures of any right limit are $\nu$. This leads naturally to

\begin{definition} \label{Khrushchev Class}
We say that a whole-line CMV matrix, $\EC$, belongs to
\emph{Khrushchev Class} if $\mu_n=\mu_m$ for all $n,\ m \in \Z$,
where $\mu_j$ is the spectral measure of $\EC$ and $\delta_j$.
\end{definition}

By the discussion above,

\begin{proposition}\label{khrushchev class proposition}
If $\cmv$ is a CMV matrix such that the sequence $\mu_n$ has a weak limit as $n
\rightarrow \infty$ then all right limits of $\cmv$ belong to Khrushchev Class.
\end{proposition}

Thus, Khrushchev theory reduces to the analysis of Khrushchev Class.
Since Simon analyzed the analogous Jacobi case \cite{Simon-Weak}, we
feel the following is fitting:

\begin{definition} \label{Simon Class}
We say that a whole-line Jacobi matrix, $H$, belongs to \emph{Simon
Class} if $\mu_n=\mu_m$ for all $n,\ m \in \Z$, where $\mu_j$ is the
spectral measure of $H$ and $\delta_j$.
\end{definition}

The final section of this paper will be devoted to the analysis of
these two classes. In particular, we rederive all of the main
results of \cite{Simon-Weak} and even extend some of those of
\cite{Khrushchev}. We conclude with an amusing (and easy) fact:

\begin{proposition} \label{Simon and Khrushchev Class spectrum}
Any $H$ in the Simon Class is either periodic, and so reflectionless
on its spectrum, or decomposes into a direct sum of finite $($in
fact $2\times2$ matrices$)$, and so has pure point spectrum of
infinite multiplicity.

Similarly, any $\EC$ in the Khrushchev Class that does not belong to
the class introduced in Example \ref{example} is either
reflectionless on its spectrum or has pure point spectrum.
\end{proposition}

The rest of this paper is structured as follows. Section 2 contains the proof
of Theorems \ref{Remling} and \ref{sparse} as well as an application to random
perturbations of CMV matrices. Section 3 contains a proof of Theorems
\ref{equivalence} and \ref{Khrushchev}, and Section 4 contains our analysis of
the operators in the Khrushchev and Simon Classes and their relevance to
Khrushchev's theory of weak limits and ratio asymptotics.

{\bf Acknowledgments.} We would like to thank Barry Simon for
helpful discussions, as well as the referees for their useful
comments.

%
%
%
%

\section{The Proof of Remling's Theorem for OPUC}
Our proof will parallel that of Remling \cite{Remling} quite
closely, so we will content ourselves with presenting the parts that
differ significantly, but only sketching those parts that are
similar.

We will first need some definitions.  Let $z \in \D$ and let $S \subset
\partial \D$ a Borel set, and define $$\omega_z(S) = \int_S \re \Bigl(
\frac{e^{i\theta}+z}{e^{i\theta}-z} \Bigr) \frac{d\theta}{2\pi}.$$
(Here, and numerous times below, we have made use of the standard
identification of $\partial \D$ with $[0,2 \pi)$ in that the
integration is actually over the set $\{\theta \in [0,2\pi) :
e^{i\theta} \in S \}$.  We trust this will not cause any confusion.)
If $f: \D \rightarrow \D$ is a Schur function, define
$$\omega_{f(e^{i\theta})}(S) = \lim_{r \uparrow 1}
\omega_{f(re^{i\theta})}(S).$$  As $z \mapsto \omega_{f(z)}(S)$ is a
non-negative harmonic function in $\D$, Fatou's Theorem implies that
this limit exists for (Lebesgue) almost every $\theta$.

Given Schur functions $f_n(z)$ and $f(z)$, we will say that $f_n$ converges to
$f$ \emph{in the sense of Pearson} if for all Borel sets $A,S \subseteq
\partial \D$,
$$
\lim_{n \rightarrow \infty} \int_A \omega_{f_n(e^{i\theta})}(S) \frac{d\theta}{2\pi} = \int_A \omega_{f(e^{i\theta})}(S) \frac{d\theta}{2\pi}.
$$
(We note here that in \cite{BP1,BP2,Remling} this mode of convergence was called \emph{convergence in value distribution}.  However, since this term had already been used in \cite{nevanlinna} for a completely different concept, we will use the above name instead.)

The next lemma relates this type of convergence to a more standard
one:

\begin{lemma}\label{convergence}
Let $f$, $f_n$, $n\in\N$, be Schur functions. Then $f_n$ converges to $f$ in the sense
of Pearson if and only if $f_n(z)$ converges to $f(z)$ uniformly on compact
subsets of $\D$.
\end{lemma}

Of course, in this case it is well-known that the associated spectral measures then converge weakly as well.

\begin{proof}
We simply sketch the proof since the full details may be found in
\cite{Remling}. For the forward implication, we may use compactness
to pick a subsequence where $g(z):=\lim_{k \rightarrow
\infty}f_{n_k}$ exists (uniformly on compact subsets of $\D$) and
defines an analytic function.  By uniqueness of limits in the sense
of Pearson, we then must have $g = f$.

For the opposite direction, one may either use spectral averaging (as in \cite{Remling}), or simply appeal to Lemma \ref{uniformity} below.
\end{proof}

The basic result behind Theorem \ref{Remling} is the following analog of a
result of Breimesser and Pearson \cite{BP1}:

\begin{theorem}\label{Pearson}
Let $\cmv$ be a half-line CMV matrix.  For all Borel sets $S
\subseteq \partial \D$ and $A \subseteq \Sigma_{\ac}(\cmv)$ we have
$$
\lim_{n \rightarrow \infty} \Biggl(  \int_A
\omega_{f_+(e^{i\theta})}(S) \frac{d\theta}{2\pi} - \int_A
\omega_{e^{i\theta}f_-(e^{i\theta})}(S^\ast) \frac{d\theta}{2\pi}
\Biggr) = 0
$$
where $S^\ast = \{ z : \overline z \in S\}$.
\end{theorem}

Assuming Theorem \ref{Pearson} for a moment, we can prove Theorem~\ref{Remling}:

\begin{proof}[Proof of Theorem \ref{Remling}]
Let $\EC$ be a right limit of $\cmv$, so there is a sequence $n_j
\uparrow \infty$ such that $\lim_{j \rightarrow
\infty}\alpha_{n+n_j}(\cmv) =\alpha_n(\EC)$ for the
corresponding sequences of Verblunsky coefficients. Thus, if
$f_\pm(z)$ are the Schur functions of $\EC$ defined by \eqref{f} for
$n=0$, then
$$
f_\pm(z,n_j) \rightarrow f_\pm(z) \,\text{ as }\, j\to\infty
$$
uniformly on compact subsets of
$\D$.  By Lemma~\ref{convergence} and Theorem~\ref{Pearson} we now
have
$$
\int_A \omega_{f_+(e^{i\theta})}(S) \frac{d\theta}{2\pi} = \int_A \omega_{e^{i\theta}f_-(e^{i\theta})}(S^\ast) \frac{d\theta}{2\pi}
$$
for all Borel sets $A \subseteq \Sigma_{\ac}(\cmv), S \subseteq
\partial \D$.

Now Lebesgue's differentiation theorem and the fact that
$\omega_{\overline z}(S^\ast) = \omega_z(S)$ shows
$$f_+(e^{i\theta}) = e^{-i\theta} \overline{f_-(e^{i\theta})}$$
almost everywhere on $\Sigma_{\ac}(\cmv)$.  Thus, $\EC$ is
reflectionless on $\Sigma_{\ac}(\cmv)$.
\end{proof}

We now turn to the proof of Theorem \ref{Pearson}.  We will need a few preparatory results.

\begin{lemma}\label{reproducing}
For any Schur function $f(z)$, Borel set $S \subseteq \partial \D$, and $z \in \D$ we have
$$
\omega_{f(z)}(S) = \int_0^{2\pi} \omega_{f(e^{i\theta})}(S) d\omega_z(e^{i\theta}).
$$
In particular, for any Borel set $A \subseteq \partial \D$,
$$
\int_A \omega_{f(r e^{i\theta})}(S) \frac{d\theta}{2\pi} = \int_0^{2\pi} \omega_{f(e^{i\theta})}(S) \omega_{re^{i\theta}}(A) \frac{d\theta}{2\pi}.
$$
\end{lemma}

\begin{proof}
For the first statement, just note that both sides are harmonic functions of
$z$ with the same boundary values.  The second statement follows by writing
$$d\omega_{re^{i\theta}}(e^{i\phi}) = \frac{1-r^2}{1+r^2-2r\cos (\phi-\theta)}
\frac{d\phi}{2\pi}$$ and applying Fubini's theorem.
\end{proof}

\begin{lemma}\label{uniformity}
Let $A \subseteq \partial \D$ be a Borel subset.  Then
$$
\lim_{r \uparrow 1} \sup_{f,S} \Biggl| \int_A \omega_{f(re^{i\theta})}(S) \frac{d\theta}{2\pi} - \int_A \omega_{f(e^{i\theta})}(S) \frac{d\theta}{2\pi} \Biggr| = 0
$$
where the supremum is taken over all Schur functions $f(z)$ and all Borel sets $S \subseteq \partial \D$.
\end{lemma}

\begin{proof}
This follows from Lemma~\ref{reproducing} and analyzing (the
$f$-independent quantity) $$\int_A \omega_{re^{i\theta}}(A^c)
\frac{d\theta}{2\pi}.$$  For more details, see Lemma A.1 in
\cite{Remling} whose proof is nearly identical.
\end{proof}

We will need a notion of pseudohyperbolic distance on $\D$.  Given $w_1,w_2 \in \D$ define
$$
\gamma(w_1 , w_2 ) = \frac{|w_1 - w_2|}{\sqrt{1-|w_1|^2}\sqrt{1-|w_2|^2}}.
$$
This is an increasing function of the hyperbolic distance on $\D$.
As such, if $F:\D \rightarrow \D$ is analytic, then $$\gamma\bigl(
F(w_1) , F(w_2) \bigr) \leq \gamma (w_1 , w_2)$$ and if $F$ is an
automorphism with respect to hyperbolic distance on $\D$
(written``$F \in \rm{Aut}(\D)$'') then we have equality above.
Taking $F(z)$ to be the analytic function whose real part is
$\omega_z(S)$, we see that for all $z,\zeta \in \D$ and all Borel
sets $S \subseteq
\partial \D$,
\begin{equation}\label{o g relation}
|\omega_{z}(S) - \omega_{\zeta}(S) | \leq \frac{|\omega_{z}(S) - \omega_{\zeta}(S) |}{ \sqrt{1 - |\omega_z(S)|^2} \sqrt{1 - |\omega_{\zeta}(S)|^2} } \leq \gamma\bigl( F(z) , F(\zeta) \bigr) \leq \gamma(z,\zeta).
\end{equation}

Now let $\{\alpha_n\}_{n \in \Z}$ be a sequence of Verblunsky coefficients (some of which may lie on $\partial \D$).  Recall the two sequences of Schur functions defined by \eqref{f}:
$$
f_+(z,n) = f(z;\alpha_n , \alpha_{n+1}, \dots), \quad\quad\quad
f_-(z,n) = f(z; -\overline{\alpha_{n-1}} , -\overline{\alpha_{n-2}}
, \dots).
$$
Since the Schur algorithm terminates at any $\alpha_k \in \partial
\D$, we see that for a half-line sequence of $\alpha$'s (recall
$\alpha_{-1} = -1$) we have $f_-(z,n=0) = -\overline{\alpha_{-1}} =
1$.

Viewing matrix arithmetic projectively (that is, identifying an
automorphism of $\D$ with its coefficient matrix, see for instance
\cite{Remling}), the Schur algorithm shows
$$f_\pm(z,n+1) = T_\pm(z,\alpha_n)f_\pm (z,n)$$ where
$$T_+(z,\alpha) = \begin{bmatrix} 1&-\alpha \\ -z \overline \alpha &
z \end{bmatrix} \quad\text{and}\quad T_-(z,\alpha) = \begin{bmatrix}
z&-\overline \alpha \\ -z \alpha & 1 \end{bmatrix}.$$
By elementary
manipulations we see that for any $z \in \C$,
\begin{equation}\label{T relation}
T_+(z,\alpha) = \begin{bmatrix} 1 & 0 \\ 0 & z \end{bmatrix} \overline{T_-(z,\alpha)} \begin{bmatrix} 1 & 0 \\ 0 & \overline z \end{bmatrix}.
\end{equation}
We will let $$P_\pm(z,n) = T_\pm (z, \alpha_{n-1}) \cdots T_\pm (z,\alpha_0)$$ so that $$f_\pm (z,n) = P_\pm (z,n) f_\pm (z,n=0).$$

We have the following mapping properties of $T_\pm(z,\alpha)$:

\begin{lemma}\label{mapping properties}
Let $\alpha \in \D$.

(1) If $z \in \partial \D$, then $T_\pm (z,\alpha) \in
\rm{Aut}(\D)$.

(2) If $z \in \D$, then $T_-(z,\alpha): \D \rightarrow \D$ and $$\gamma \bigl(T_-(z,\alpha)w_1 , T_-(z,\alpha)w_2 \bigr) \leq |z| \gamma (w_1 , w_2)$$ for all $w_1,w_2 \in \D$.
\end{lemma}

\begin{proof}
Let $$S(\alpha) = \begin{bmatrix} 1 & -\alpha \\ -\overline \alpha &
1\end{bmatrix} \quad\text{and}\quad M(z) = \begin{bmatrix} z & 0 \\
0 & 1 \end{bmatrix}$$ so that $$T_+(z,\alpha) = M(z^{-1})S(\alpha)
\quad\text{and}\quad T_-(z,\alpha) = S(\overline \alpha) M(z).$$
Because $\alpha \in \D$ we have $S(\alpha) \in \rm{Aut}(\D)$.  If $z
\in
\partial \D$ then $M(z) \in \rm{Aut}(\D)$ as well, while a
straightforward calculation shows that if $z \in \D$ then $$\gamma
\bigl( M(z)w_1 , M(z)w_2 \bigr) \leq |z| \gamma(w_1 , w_2).$$  This
proves (1) and (2).
\end{proof}

With these preliminaries in hand we are ready for the proof of
Theorem~\ref{Pearson}.  We emphasize again that we are following the
proof of Theorem 3.1 from \cite{Remling}.

\begin{proof}[Proof of Theorem \ref{Pearson}]
Subdivide $A = A_0 \cup A_1 \cup \dots \cup A_N$ in such a way that
\begin{enumerate}[1.]
\item $|A_0| < \eps$.
\item On $\bigcup_{k=1}^N A_k$, $\lim_{r \uparrow
1}f_+(re^{i\theta},0)$ exists and lies in $\D$.
\item For each $1 \leq k \leq N$, there is a point $m_k \in \D$ such that $\gamma \bigl( f_+(e^{i\theta},0), m_k \bigr) < \eps$ for all $e^{i\theta}\in A_k$.
\end{enumerate}
The construction of such a decomposition is identical to that given
in \cite{Remling}, so we do not review it here.

To deal with $A_0$, we note that for any $z \in \overline \D$ and any Borel set $S \subseteq \partial \D$, we have $|\omega_z(S)| \leq 1$.  Thus
$$
\Biggl| \int_{A_0} \omega_{f_+(e^{i\theta},n)}(S) \frac{d\theta}{2\pi} - \int_{A_0} \omega_{e^{i\theta}f_-(e^{i\theta},n)}(S^\ast) \frac{d\theta}{2\pi}  \Biggr| < 2\eps.
$$

Now we consider $A_1, \dots, A_N$.  Notice that if $e^{i\theta} \in
\bigcup_{k=1}^N A_k$, then for all $n \in \N$ we also have that
$\lim_{r \uparrow 1}f_+(re^{i\theta},n)$ exists and lies in $\D$. As
$P_+(e^{i\theta},n) \in \rm{Aut}(\D)$ we see $$\gamma\bigl(
f_+(e^{i\theta},n) , P_+(e^{i\theta},n) m_k \bigr) < \eps$$ for all
$e^{i\theta} \in A_k$ and all $n \in \N$. Using \eqref{o g relation}
and integrating we find
$$
\Biggl| \int_{A_k} \omega_{f_+(e^{i\theta},n)}(S) \frac{d\theta}{2\pi} - \int_{A_k} \omega_{P_+(e^{i\theta},n) m_k}(S) \frac{d\theta}{2\pi} \Biggr| < \eps |A_k|.
$$

By \eqref{T relation} and the fact that $\omega_{\overline
z}(S^\ast) = \omega_z (S)$, we can rewrite this as
\begin{equation}\label{e1}
\Biggl| \int_{A_k} \omega_{f_+(e^{i\theta},n)}(S) \frac{d\theta}{2\pi} - \int_{A_k} \omega_{e^{i\theta}P_-(e^{i\theta},n) (e^{-i\theta} m_k)}(S^\ast) \frac{d\theta}{2\pi} \Biggr| < \eps |A_k|
\end{equation}
(and notice that because $T_-(z,\alpha) = S(\overline \alpha) M(z)$, we have that $z P_-(z,n) (z^{-1} m_k)$ is indeed a Schur function).

By Lemma \ref{mapping properties} there is an $n_0 \in \N$ so that for all $n \geq n_0$, $$\gamma \bigl( zP_-(z,n) (z^{-1}w_k) , z f_-(z,n) \bigr) < \eps.$$  As before, using \eqref{o g relation} and integrating shows
\begin{equation}\label{e2}
\Biggl| \int_{A_k} \omega_{zP_-(z,n) (z^{-1}w_k)}(S^\ast) \frac{d\theta}{2\pi}  -  \int_{A_k} \omega_{zf_-(z,n)}(S^\ast) \frac{d\theta}{2\pi} \Biggr| < \eps |A_k|.
\end{equation}

Now use Lemma \ref{uniformity} to find an $r<1$ so that
$$
\Biggl| \int_{A_k} \omega_{f(e^{i\theta})}(S) \frac{d\theta}{2\pi} - \int_{A_k} \omega_{f(re^{i\theta})}(S) \frac{d\theta}{2\pi}  \Biggr| < \eps |A_k|
$$
for all Schur functions $f(z)$, all Borel sets $S \subseteq \partial \D$, and $k=1, \dots, N$.  Applying this to \eqref{e1} and \eqref{e2} shows
$$
\Biggl| \int_{A_k} \omega_{f_+(e^{i\theta},n)}(S) \frac{d\theta}{2\pi} -  \int_{A_k} \omega_{e^{i\theta}f_-(e^{i\theta},n)}(S^\ast) \frac{d\theta}{2\pi}  \Biggr| < 4\eps |A_k|.
$$

Now summing in $k$ shows
$$
\Biggl| \int_{A} \omega_{f_+(e^{i\theta},n)}(S) \frac{d\theta}{2\pi} -  \int_{A} \omega_{e^{i\theta}f_-(e^{i\theta},n)}(S^\ast) \frac{d\theta}{2\pi}  \Biggr| < 4\eps |A| + 2\eps
$$
for all $n \geq n_0$.
\end{proof}

Next, we illustrate Theorem \ref{Remling} by a simple example of
constant coefficients CMV matrices:
\begin{example}
Let $\cmv$ be the half-line CMV matrix associated with the constant Verblunsky coefficients $\al_n=a$, $n\geq0$, for some $a\in(0,1)$. It follows from the Schur algorithm that the corresponding Schur function $f_a$ satisfies the quadratic equation
\[
azf_a(z)^2 + (1-z)f_a(z) - a = 0,
\]
and hence is given by
\[
f_a(z)=\f{-(1-z)+\sqrt{(1-z)^2+4a^2z}}{2az}, \quad z\in\D,
\]
where the square root is defined so that
$\sqrt{e^{i\te}}=e^{i\te/2}$ for $\te\in(-\pi,\pi)$. Using the
Carath\'eodory function $F_a(z)=\f{1+zf_a(z)}{1-zf_a(z)}$ we compute
\begin{align*}
\Sigma_{\ac}(\cmv) &= \{e^{i\te} \;:\; \Re\,F_a(e^{i\te})>0 \} = \{e^{i\te} \;:\; |f_a(e^{i\te})|<1\} \\ &= \{e^{i\te} \;:\; 2\arcsin(a)<\te<2\pi-2\arcsin(a)\}.
\end{align*}
The half-line CMV matrix $\cmv$ has exactly one right limit $\EC$ which is the whole-line CMV matrix associated with the constant coefficients $\wti\al_n=a$, $n\in\Z$. It follows from \eqref{f} that the two Schur functions for $\EC$ are given by $f_+(z,n)=f_a(z)$ and $f_-(z,n)=f_{-a}(z)=-f_a(z)$, $n\in\Z$. Since for all $e^{i\te}\in\Sigma_{\ac}(\cmv)$,
\[
f_a(e^{i\te})=\f{i\sin(\te/2)}{ae^{i\te/2}} \left(1-\sqrt{1-\left(\f{a}{\sin(\te/2)}\right)^2}\,\right)
\]
and the expression under the square root is positive, one easily verifies the reflectionless property of $\EC$ on $\Sigma_{\ac}(\cmv)$,
\[
e^{i\te}f_+(e^{i\te},n)=e^{i\te}f_a(e^{i\te})=-\ol{f_a(e^{i\te})}=\ol{f_-(e^{i\te},n)}, \quad e^{i\te}\in\Sigma_{\ac}(\cmv),
\]
thus confirming the claim of Theorem \ref{Remling}.

Note that adding a decaying perturbation to the Verblunsky
coefficients of $\cmv$ does not change the uniqueness of the right
limit, nor does it change the limiting operator. Moreover, if the
decay is sufficiently fast (e.g.\ $\ell^1$), $\Sigma_{\ac}(\cmv)$
does not change either.
\end{example}

The following is one of the reasons reflectionless operators are so useful:

\begin{lemma}\label{deterministic}
Let $\{\alpha_n\}_{n \in \Z}$, $\{\beta_n\}_{n \in \Z}$ be two sequences of
Verblunsky coefficients such that their corresponding whole-line CMV matrices
are both reflectionless on some common set $A$ with $|A|>0$. If
$\alpha_n=\beta_n$ for all $n <0$, then $\alpha_n=\beta_n$ for all $n$.
\end{lemma}

\begin{proof}
By the Schur algorithm, $\{\alpha_n\}_{n<0}$ determines $f_-(z,0)$. This, by
Definition \ref{reflectionless}, determines $f_+(z,0)$ on $A$. But the values
of a Schur function on a set of positive Lebesgue measure on $\partial \D$
determine the Schur function. Thus, $f_+(z,0)$ is determined throughout $\D$ by
$\{\alpha_n\}_{n<0}$. But, by the Schur algorithm again, this determines
$\{\alpha_n\}_{n \geq 0}$.
\end{proof}

\begin{proof}[Proof of Theorem \ref{sparse}]
Take a subsequence $\{n_{j_k} \}_{k=1}^\infty$, of $n_j$, such that both
$$\lim_{k \rightarrow \infty} \alpha_{n+n_{j_k}}\equiv \beta_n$$ and
$$\lim_{k \rightarrow \infty}  \widetilde{\alpha}_{n+n_{j_k}} \equiv \widetilde{\beta}_n$$
exist for every $n \in \Z$. By conditions $(ii)$ and $(iii)$ of the
theorem, $\beta_n=\widetilde{\beta}_n$ for all $n<0$ but $\beta_0
\neq \widetilde{\beta}_0$. By Theorem~\ref{Remling} the
corresponding whole-line CMV matrices, $\EC$ and $\widetilde{\EC}$,
are reflectionless on $\Sigma_{\ac}(\cmv)$ and
$\Sigma_{\ac}(\widetilde{\cmv})$ respectively. Thus, by
Lemma~\ref{deterministic} these two sets cannot intersect each
other.

Viewing a CMV matrix satisfying \eqref{mate-nevai-sparse} as a
perturbation of the CMV matrix with all Verblunsky coefficients
equal to zero (this matrix has spectral measure
$\frac{d\theta}{2\pi}$), we see by the above analysis that such a
matrix cannot have any absolutely continuous spectrum, since
\eqref{mate-nevai-sparse} is easily seen to imply conditions $(ii)$
and $(iii)$ of the theorem.
\end{proof}

We conclude this section with an application of Theorem
\ref{Remling} to random CMV matrices.

\begin{theorem}\label{random}
Let $\{\beta_n(\omega)\}_{n=1}^\infty$ be a sequence of random
Verblunsky coefficients of the form:
$$\beta_n(\omega)=\alpha_n+s_n X_n(\omega)$$ where
$X_n(\omega)$ is a sequence of independent, identically distributed
random variables whose common distribution is not supported at a
single point, and $s_n$ is a bounded sequence such that
$|\alpha_n+s_n|<1$. Let $\cmv(\omega)$ be the corresponding random
CMV matrix. If $s_n \not \rightarrow 0$ as $n \rightarrow \infty$
then $\Sigma_{\ac}(\cmv(\omega))= \emptyset$ almost surely.
\end{theorem}

We first need a lemma:
\begin{lemma}\label{deterministic-ac}
Let $\{\beta_n(\omega)\}_{n=1}^\infty$ be a sequence of independent
random Verblunsky coefficients and let $\cmv(\omega)$ be the
corresponding family of CMV matrices. Then there exists a set $A
\subseteq \partial \D$ such that with probability one,
$\Sigma_{\ac}(\cmv(\omega))=A$.
\end{lemma}
\begin{proof}
This is the CMV version of a theorem of Jak\v si\'c and Last
\cite[Cor.\ 1.1.3]{Jak-Last} for Jacobi matrices. The proof is the
same: Since the absolutely continuous spectrum is stable under
finite rank perturbations, it is easily seen to be a tail event.
Thus, the result is implied by Kolmogorov's 0-1 Law. For details see
\cite{Jak-Last}.
\end{proof}

\begin{proof}[Proof of Theorem \ref{random}]
Pick a sequence, $n_j$, such that $\lim_{j \rightarrow \infty}
s_{n+n_j} \equiv S_n$ exists for every $n \in \Z$ and $S_0 \neq 0$. Such a
sequence exists by the assumptions on $s_n$. By restricting to a
subsequence, we may assume that $\lim_{j \rightarrow
\infty}\alpha_{n+n_j}+s_{n+n_j}=\widetilde{\alpha}_n+S_n$ also
exists for any $n \in \Z$.

Let $X_1 \neq X_2$ be two points in the support of the common distribution of
$X_n(\omega)$. By the Borel-Cantelli Lemma, with probability one, there exist
two subsequences $n_{j_k(\omega)}$ and $n_{j_l(\omega)}$, such that $$\lim_{k
\rightarrow \infty}X_{n+n_{j_k(\omega)}}(\omega)= X_1,  \quad n \in \Z$$ and
$$\lim_{l \rightarrow \infty}X_{n+n_{j_l(\omega)}}(\omega)=\begin{cases}
X_1 & n <0, \\
 X_2 & n \geq 0.
\end{cases}
$$

Let $\EC_1$ and $\EC_2$ be the two whole-line CMV matrices corresponding to the
sequences $$\widetilde{\beta}_n^1 \equiv \widetilde{\alpha}_n+S_nX_1, \quad n
\in \Z$$ and
$$
\widetilde{\beta}_n^2=\begin{cases}
\widetilde{\alpha}_n+S_nX_1 & n <0,\\
\widetilde{\alpha}_n+S_nX_2 & n \geq 0.
\end{cases}
$$
respectively. If it were not true that $\Sigma_{\ac}(\cmv(\omega))=\emptyset$
almost surely, then by Lemma \ref{deterministic-ac}, the essential support of
the absolutely continuous spectrum would be some deterministic set $A \neq
\emptyset$. By Theorem~\ref{Remling}, since $\EC_1$ and $\EC_2$ are both right
limits of CMV matrices with absolutely continuous spectrum on $A$, they are
both reflectionless on $A$. This contradicts Lemma~\ref{deterministic} so we
see that $\Sigma_{\ac}(\cmv(\omega))= \emptyset$ almost surely.
\end{proof}

\section{Reflectionless Matrices and Reflectionless Measures}

In this section we verify that Example~\ref{example} is the only
example where the two notions of reflectionless (cf.\ Definitions
\ref{reflectionless} and \ref{measure-reflectionless}) are not
equivalent. As an application of this fact we give a short proof of
Theorem \ref{Khrushchev}, a special case of Khrushchev's results.

We start by introducing the whole-line unitary 5-diagonal CMV matrix $\EC$ associated to $\{\alpha_n\}_{n \in \Z}$ by
\begin{align} \lb{whole-line-cmv}
\EC = \begin{pmatrix} \ddots &&\hspace*{-8mm}\ddots
&\hspace*{-10mm}\ddots &\hspace*{-12mm}\ddots
&\hspace*{-14mm}\ddots &&&
\raisebox{-3mm}[0mm][0mm]{\hspace*{-6mm}{\Huge $0$}}
\\
&0& \ol\al_{0}\rho_{-1} & -\ol\al_{0}\al_{-1} &
\ol\al_{1}\rho_{0} & \rho_{1}\rho_{0}
\\
&& \rho_{0}\rho_{-1} &-\rho_{0}\al_{-1} &
-\ol\al_{1}\al_{0} & -\rho_{1}\al_{0} & 0
\\
&&&0& \ol\al_{2}\rho_{1} & -\ol\al_{2}\al_{1} &
\ol\al_{3}\rho_{2} & \rho_{3}\rho_{2}
\\
&&\raisebox{-5mm}[0mm][0mm]{\hspace*{-6mm}{\Huge $0$}} && \rho_{2}\rho_{1} &
-\rho_{2}\al_{1} & -\ol\al_{3}\al_{2} & -\rho_{3}\al_{2} & 0
\\
&&&&&\hspace*{-14mm}\ddots &\hspace*{-14mm}\ddots
&\hspace*{-14mm}\ddots &\hspace*{-8mm}\ddots &\ddots
\end{pmatrix},
\end{align}
where $\rho_n = (1 - |\alpha_n|^2)^{1/2}$.  Here the diagonal
elements are given by $\EC_{n,n}=-\ol\al_n\al_{n-1}$. It is known
\cite{CMV, OPUC1} that CMV matrices have the following $\cL\cM$
factorization:
\begin{align} \lb{LM}
\EC = \cL\cM,
\end{align}
where
\begin{equation*}
\cL = \begin{pmatrix} \ddots & & &
\raisebox{-3mm}[0mm][0mm]{\hspace*{-5mm}\Huge $0$}  \\
& \te_{2k-2} & & \\ & & \te_{2k} & & \\
& \raisebox{0mm}[0mm][0mm]{\hspace*{-10mm}\Huge $0$} & & \ddots
\end{pmatrix}, \quad\quad
\cM = \begin{pmatrix} \ddots & & &
\raisebox{-3mm}[0mm][0mm]{\hspace*{-5mm}\Huge $0$} \\
& \te_{2k-1} &  &  \\ &  & \te_{2k+1} &  & \\
& \raisebox{0mm}[0mm][0mm]{\hspace*{-10mm}\Huge $0$} & & \ddots
\end{pmatrix},
\end{equation*}
such that, for all $k \in \Z$
\begin{gather*}
\begin{pmatrix}
\cL_{2k,2k} & \cL_{2k,2k+1} \\ \cL_{2k+1,2k} & \cL_{2k+1,2k+1}
\end{pmatrix} =  \te_{2k},
\quad\quad
\begin{pmatrix}
\cM_{2k-1,2k-1} & \cM_{2k-1,2k} \\ \cM_{2k,2k-1} & \cM_{2k,2k}
\end{pmatrix} =  \te_{2k-1},
\\
\te_k = \begin{pmatrix}\ol\al_k & \rho_k \\ \rho_k & -\al_k\end{pmatrix}.
\end{gather*}

Next, we introduce the diagonal Schur function $f(z,n)$ associated with the
diagonal spectral measure $\mu_n$ (i.e., the spectral measure of $\EC$ and
$\de_n$) by
\begin{equation*}
\f{1+zf(z,n)}{1-zf(z,n)} =
\int_{0}^{2\pi}\f{e^{i\te}+z}{e^{i\te}-z}\,d\mu_n(\te) =
\langle\de_n,(\EC+zI)(\EC-zI)^{-1}\de_n\rangle, \quad n\in\Z.
\end{equation*}
Then it follows from Definition \ref{measure-reflectionless} that the measure
$\mu_n$ is reflectionless on $A\subseteq\dD$ if and only if $\Im(zf(z,n))=0$
for a.e.\ $z\in A$.  Recall from \eqref{diag-schur} that the diagonal and half-line Schur functions are related by
$$
f(z,n)=f_+(z,n)f_-(z,n).
$$
Thus, if a whole-line CMV matrix, $\EC$, is reflectionless on $A$,
it follows that all diagonal measures $\mu_n$, $n\in\Z$, are
reflectionless on $A$ as well. As indicated in
Example~\ref{example}, the converse is not true in general.
Nevertheless, one can show that this example is the only exceptional
case: whenever all $\al$'s are identically zero, or there are at
least two nonzero $\al$'s, the converse holds. We start by showing
that a reflectionless CMV matrix $\EC$ with a finite number of
nonzero $\al$'s can have no more than one nonzero Verblunsky
coefficient.

\begin{lemma} \lb{many-al}
Suppose $\EC$ is a whole-line CMV matrix of the form \eqref{whole-line-cmv}
such that $\al_m\neq0$, $\al_n\neq0$ for some $m,n\in\Z$, $m<n$, and $\mu_n$ is
reflectionless on a set $A\subseteq\dD$ of positive Lebesgue measure. Then
there are infinitely many nonzero $\al$'s.
\end{lemma}
\begin{proof}
Assume that there are finitely many nonzero $\al$'s. Then it follows
from the Schur algorithm that the Schur functions $f_\pm(z,n)$ are
rational functions of $z$ with finitely many zeros in $\ol\D$ and
poles in $\C\setminus\ol\D$. By \eqref{diag-schur} the same holds
for $f(z,n)$.

Let $B$ be a neighborhood of the \emph{finite} set of zeros of
$f(z,n)$ on $\partial \D$ such that $A \setminus B$ has positive
Lebesgue measure. Then $\log(zf(z,n))$ is a well-defined analytic
function on some open neighborhood of $\dD \setminus B$.

The reflectionless assumption implies that
$\Im\log(zf(z,n))\in\{0,\pi\}$ Lebesgue a.e.\ on $A \setminus B$,
and hence by the Cauchy--Riemann equations, the analytic function
$\f{d}{dz}\log(zf(z,n))$ is zero on accumulation points of $A
\setminus B$. Since $A \setminus B$ is of positive Lebesgue measure,
the set of its accumulation points is also of positive Lebesgue
measure, and hence $\f{d}{dz}\log(zf(z,n))$ is identically zero in
the neighborhood of $\dD \setminus B$. This implies that $zf(z,n)$
is a nonzero constant in $\D \setminus B$. This is a contradiction
since $f(z,n)$ is analytic in $\D$.
\end{proof}

\begin{proof}[Proof of Theorem~\ref{equivalence}]
Assume $\EC$ is reflectionless in the sense of Definition \ref{reflectionless}.
Then it follows from the discussion at the beginning of this section that
$\mu_n$ is reflectionless for all $n$.

Now assume $\mu_n$ is reflectionless in the sense of Definition
\ref{measure-reflectionless} for all $n$. Since two $\alpha$'s are not zero, it
follows from Lemma~\ref{many-al} that there are infinitely many nonzero
$\al$'s. Let $\al_{n_{-1}}$, $\al_{n_0}$, $\al_{n_1}$, $\al_{n_2}$, denote four
consecutive nonzero $\al$'s, that is, four non-zero values with possibly some
zero values between them.  For $n\in\Z$, introduce $g_+(z,n)=zf_+(z,n)$ and
$g_-(z,n)=\ol{f_-(z,n)}$. Then $g_\pm(z,n_j)$, $j=0,1,2$, are
not identically zero functions and it follows from the Schur algorithm that for
$z\in\dD$ and $n\in\Z$,
\begin{align} \lb{recursion}
g_\pm(z,n-1) = z\frac{g_\pm(z,n)+\al_{n-1}}{\ol\al_{n-1}g_\pm(z,n)+1}, \quad
g_\pm(z,n+1) = \frac{g_\pm(z,n)/z-\al_n}{-\ol\al_{n}g_\pm(z,n)/z+1}.
\end{align}

The reflectionless condition at $n_1$ implies
$g_+(z,n_1)\ol{g_-(z,n_1)}=zf(z,n_1)\in\R\setminus\{0\}$ for a.e.\
$z\in A$, and hence $g_\pm(z,n_1)=s_\pm(z)e^{it(z)}$ with
$s_\pm(z)\in\R$ and $t(z)\in[0,\pi)$ for a.e.\ $z\in A$. In order to
check that $\EC$ is reflectionless in the sense of Definition
\ref{reflectionless} it remains to check that $s_+(z)=s_-(z)$ for
a.e.\ $z\in A$.

By construction all $\al$'s between $\al_{n_0}$ and $\al_{n_1}$ are zero, hence
it follows from \eqref{recursion} that
\begin{equation*}
g_\pm(z,n_0) = z^{n_1-n_0}e^{it(z)}
\f{s_\pm(z)+e^{-it(z)}\al_{n_0}}{\ol\al_{n_0}e^{it(z)}s_\pm(z)+1}
\quad\text{for a.e.\ $z\in A$.}
\end{equation*}

Since for every $\ga\in\D\setminus\R$ the function
$h_\ga:[-1,1]\to(-\f{\pi}{2},\f{\pi}{2})$ defined by
\begin{align}\lb{arg-fn}
h(x)=\arg\left(\ol\ga\f{x+\ga}{\ol\ga x+1}\right)
\end{align}
is $1-1$, \eqref{recursion} and the reflectionless condition at
$n_0$ (i.e.,
$g_+(z,n_0)\ol{g_-(z,n_0)}=zf(z,n_0)\in\R\setminus\{0\}$ for a.e.\
$z\in A$) imply
\begin{equation}\lb{n0}
s_+(z)=s_-(z) \quad\text{for a.e.\ $z\in A$ such that
$e^{-it(z)}\al_{n_0}\in\D\setminus\R$.}
\end{equation}
Similarly, since by construction all $\al$'s between $\al_{n_1}$ and
$\al_{n_2}$ are zero, it follows from \eqref{recursion} that
\begin{equation*}
g_\pm(z,n_2) =
z^{n_1-n_2}e^{it(z)}\frac{s_\pm(z)-e^{-it(z)}z^{n_2-n_1}\al_{n_1}}
{-\ol\al_{n_1}e^{it(z)}z^{n_1-n_2}s_\pm(z)+1} \quad\text{for a.e.\ $z\in A$,}
\end{equation*}
and hence, the reflectionless condition at $n_2$ (i.e.,
$g_+(z,n_2)\ol{g_-(z,n_2)}=zf(z,n_2)\in\R\setminus\{0\}$ for a.e.\
$z\in A$), together with the injectivity of $h_\ga$, implies
\begin{equation} \lb{n2}
s_+(z)=s_-(z) \quad\text{for a.e.\ $z\in A$ such that
$e^{-it(z)}z^{n_2-n_1}\al_{n_2}\in\D\setminus\R$.}
\end{equation}
Since $e^{-it(z)}\al_{n_0}\in\R$ and $e^{-it(z)}z^{n_2-n_1}\al_{n_2}\in\R$ may
hold simultaneously only on a finite set, it follows from \eqref{n0} and
\eqref{n2} that
\begin{equation*}
s_+(z)=s_-(z) \quad\text{for a.e.\ $z\in A$},
\end{equation*}
and hence $g_+(z,n_1)=g_-(z,n_1)$ for a.e.\ $z\in A$.  That is, $\EC$ is
reflectionless on $A$ according to Definition \ref{reflectionless}.
\end{proof}

\begin{remark} We note that the case of identically zero Verblunsky
coefficients corresponds to $f_\pm(z,n)\equiv0$ for all $n\in\Z$, so that the
associated CMV matrix is reflectionless on $\dD$ in the sense of Definition
\ref{reflectionless} and hence all its diagonal spectral measures are
reflectionless on $\dD$ in the sense of Definition
\ref{measure-reflectionless}. The case of a single nonzero coefficient,
discussed in Example~\ref{example}, corresponds to one of $f_+(z,n)$ or
$f_-(z,n)$ being nonzero and the other being identically zero for each
$n\in\Z$, so that the associated CMV matrix is not reflectionless on any subset
of $\dD$ of positive Lebesgue measure, yet all the diagonal measures are
reflectionless on $\dD$.
\end{remark}

\begin{proof}[Proof of Theorem~\ref{Khrushchev}]
Let $\EC$ be a right limit of $\cmv$. Then all the diagonal measures
of $\EC$ are identical and equal to $\f{d\te}{2\pi}$. The
corresponding diagonal Schur functions in this case are
$f(z,n)\equiv0$ for all $n\in\Z$. Hence by \eqref{diag-schur} for
each $n\in\Z$ either $f_-(z,n)\equiv0$ or $f_+(z,n)\equiv0$ or both.
The latter case corresponds to $\EC$ having identically zero
Verblunsky coefficients and the other two cases correspond to $\EC$
having exactly one nonzero Verblunsky coefficient. Since this holds
for all right limits of $\cmv$, we conclude that for all $k\in\N$,
\begin{align} \lb{n+k}
\lim_{n \rightarrow \infty}\alpha_n \alpha_{n+k}=0.
\end{align}
Conversely, \eqref{n+k} implies that all right limits of $\cmv$ may have at
most one nonzero Verblunsky coefficient. Hence all right limits of $\cmv$ have
identical diagonal spectral measures equal to $\f{d\te}{2\pi}$. This implies
that the diagonal measures $d\mu_n$ of $C$ converge weakly to $\f{d\te}{2\pi}$
as $n\to\infty$.

The final statement follows since if $\alpha_n \not \rightarrow 0$ then clearly
\eqref{mate-nevai-sparse} holds, which implies, by Theorem \ref{sparse}, that
$\mu$ is purely singular.
\end{proof}

\section{The Simon and Khrushchev Classes}

In this section we extend the discussion of the previous section to
include all cases where $\mu_n$ has a weak limit. We consider both
the Jacobi and CMV cases, but begin with the Jacobi case since it is
technically simpler.

\subsection{The Simon Class}

A half-line Jacobi matrix is a semi-infinite matrix of the form:
\begin{equation} \label{half-line-Jacobi} J \left(
\{a_n,b_n\}_{n=1}^\infty \right)=\left(
\begin{array}{ccccc}
b_1    & a_1 & 0      & 0      & \dots \\
a_1    & b_2 & a_2    & 0      & \dots \\
0      & a_2 & b_3    & a_3    & \dots \\
\dots & \dots   & \dots & \dots & \dots \\
\end{array} \right).
\end{equation}

Its whole-line counterpart is defined by
\begin{equation}\lb{whole-line-Jacobi}
H \left( \{a_n,b_n\}_{n\in \Z} \right) = \left (
\begin{array}{ccccccc} \ddots & \ddots & \ddots & & & &
\raisebox{-3mm}[0mm][0mm]{\hspace*{-5mm}\Huge $0$}\\
& a_{-1} & b_0 & a_0 \\
& & a_0 & b_1 & a_1 \\
& & & a_1 & b_2 & a_2 \\
& \raisebox{3mm}[0mm][0mm]{\hspace*{-10mm}\Huge $0$}
& & & \ddots & \ddots & \ddots \\
\end{array} \right), \end{equation}
(we assume $b_n \in \R, \ a_n\geq0 $ in both cases).

These matrices may be viewed as operators on $\ell^2(\N)$ and
$\ell^2(\Z)$, respectively, and both are clearly symmetric. As is
well known, the theory of half-line Jacobi matrices with $a_n>0$ is
intimately related to that of orthogonal polynomials on the real
line (OPRL) through the recursion formula for the polynomials (see
e.g.\ \cite{Simon-Moment}). In particular, in the self-adjoint case
(to which we shall henceforth restrict our discussion), $\mu$---the
spectral measure for $J$ and $\delta_1$---is the unique solution to
the corresponding moment problem.  Furthermore,
$d\mu_n(x)=|p_{n-1}(x)|^2 d\mu(x)$ is the spectral measure for $J$
and $\delta_n$, where $\{p_n(x)\}$ are the orthonormal polynomials
corresponding to $\mu$. The whole-line matrices enter naturally into
this framework as right limits (see e.g.\ \cite{Remling} for the
definition).

The problem at the center of our discussion is that of the identification of
right limits of half-line Jacobi matrices with the property that $\mu_n$ has a
weak limit as $n \rightarrow \infty$. As is clear from the discussion in the
introduction, all these right limits belong to Simon Class (recall Definition
\ref{Simon Class}).

\begin{theorem}\label{Simon Class Characterization}
Let $H\left( \{a_n,b_n\}_{n\in \Z} \right)$ be a whole-line Jacobi matrix. The following are equivalent: \\
$(i)$ $H$ belongs to Simon Class. \\
$(ii)$ For all $m,n \in \Z$, $ \int x d\mu_n(x)=\int x d\mu_m(x)$
and $ \int x^2 d\mu_n(x)=\int x^2
d\mu_m(x)$.\\
$(iii)$ $a_{2n}=a, \ a_{2n+1}=c, \ b_n=b$ for some numbers, $a,c \geq 0$
and $b \in \R$.
\end{theorem}
\begin{remark}
In particular, this shows that if $H$ has constant first and second
moments, then $H$ belongs to Simon Class. Note, however, that the
values of these moments do not determine the element of the class
itself (not even up to translation; see \eqref{first-moment} and
\eqref{second-moment} below). Thus, it makes sense to define
$\mathcal{S}(A,B)$ to be the set of all matrices in the Simon Class
having $a^2+c^2=A$ and $b=B$, where $a,b,c$ are as in $(iii)$ above.
\end{remark}

\begin{proof}
$(i) \Rightarrow (ii)$ is trivial. \\
$(ii) \Rightarrow (iii)$: Noting that
\begin{equation} \label{first-moment}
\int x d\mu_n(x)=b_n
\end{equation} and
\begin{equation} \label{second-moment}
\int x^2 d\mu_n(x)=a_{n-1}^2+b_n^2+a_n^2,
\end{equation}
the result follows immediately
for the diagonal elements. With this in hand, comparing the second
moment of $\mu_n$ and $\mu_{n+1}$ we see that
$a_{n-1}^2+a_n^2=a_n^2+a_{n+1}^2$ from which it follows that
$a_{n-1}=a_{n+1}$, and we are done.

$(iii) \Rightarrow (i)$: Clearly, by symmetry, $\mu_{n}=\mu_{n+2}$
for all $n$, and if $a=c$ also $\mu_n=\mu_{n+1}$. Thus, we are left with showing $\mu_0=\mu_1$ under
the assumption $a \neq c$ (so at least one is nonzero). We shall
show that for any $z \in \C_+$, $$\int \frac{d\mu_0(x)}{x-z}=\int
\frac{d\mu_1(x)}{x-z}.$$

Fix $z \in \C_+$ and let $\{ u(n) \}$ be a sequence satisfying
$a_n u(n+1)+b_n u(n)+a_{n-1}u(n-1)=zu(n)$ that is $\ell^2$ at
$\infty$. This sequence is unique up to a constant factor. By the
symmetry of $H$, note that $v(n) \equiv u(1-n)$ satisfies the same
equation and is $\ell^2$ at $-\infty$. Now write
\begin{equation}
\nonumber \begin{split} \int \frac{d\mu_0(x)}{x-z}&=\langle
\delta_0, \left( H-z \right)^{-1} \delta_0 \rangle =
\frac{u(0)v(0)}{v(1)u(0)-v(0)u(1)} \\
&=
\frac{u(0)u(1)}{v(1)u(0)-v(0)u(1)}=\frac{v(1)u(1)}{v(1)u(0)-v(0)u(1)}
\\
& =\langle \delta_1, \left( H-z \right)^{-1} \delta_1 \rangle =\int
\frac{d\mu_1(x)}{x-z}, \end{split}
\end{equation}
from which, by standard results, it follows that $\mu_0=\mu_1$.
\end{proof}

We immediately get the following:

\begin{corollary}\label{generalized-Simon}
Let $J$ be a self-adjoint half-line Jacobi matrix and let $\mu$ be its spectral
measure. For $n \geq 1$, let $d\mu_n(x) = |p_{n-1}(x)|^2d\mu(x)$ be the
spectral measure of $J$ and $\delta_n$. If
\begin{equation} \label{first-moment-convergence}
\lim_{n \rightarrow \infty} \int x d\mu_n(x)=\tilde{B}
\end{equation}
\begin{equation} \label{second-moment-convergence}
\lim_{n \rightarrow \infty} \int x^2 d\mu_n(x)=\tilde{A}
\end{equation}
then $J$ is bounded and all right limits of $J$ are in $\mathcal{S}
(\tilde{A}-\tilde{B}^2,\tilde{B})$.
\end{corollary}

\begin{proof}
That $J$ is bounded follows from \eqref{first-moment} and \eqref{second-moment} applied to $J$, together
with the fact that these moments converge. Since all right limits of $J$ have constant first and second moments, it
follows from Theorem~\ref{Simon Class Characterization} that they all belong to Simon Class.
The rest follows from combining \eqref{first-moment},
\eqref{second-moment}, \eqref{first-moment-convergence} and \eqref{second-moment-convergence} together with
the definition of $\mathcal{S}(A,B)$.
\end{proof}

Corollary \ref{generalized-Simon} is closely related to Theorem 2 in
\cite{Simon-Weak}---both our assumptions and conclusions are weaker.
We also note that our proof is not that much different from the
corresponding parts in Simon's proof. However, we believe that the
``right limit point of view'' makes various ideas especially
transparent and clear. In particular, we would like to emphasize the
following subtle point.  While convergence of the first and second
moments does not imply weak convergence, by Corollary
\ref{generalized-Simon} it does imply a certain weak form of weak
convergence: it holds along any subsequence on which $J$ has a right
limit.

Also, it is now clear that any additional condition forcing weak
convergence of $\mu_n$ is equivalent to a condition that
distinguishes a particular element of
$\mathcal{S}(\tilde{A}-\tilde{B}^2,\tilde{B})$ (up to a shift).
Computing powers of $J$ shows that the third moment is not enough,
but the fourth moment is. Thus we get
\begin{theorem}[Simon \cite{Simon-Weak}]
If \eqref{first-moment-convergence} and
\eqref{second-moment-convergence} hold and $\lim_{n \rightarrow
\infty} \int x^4 d\mu_n(x)$ exists, then $J$ is bounded, $d\mu_n$
converge weakly and $J$ has a unique right limit $($up to a shift$)$
in $\mathcal{S}(\tilde{A}-\tilde{B}^2,\tilde{B})$.
\end{theorem}

We next want to demonstrate how ratio asymptotics also fit naturally into this
framework.
\begin{definition}
Let $\mu$ be a probability measure on the real line. We say $\mu$ is
\emph{ratio asymptotic} if
\begin{equation} \nonumber
\lim_{n \rightarrow \infty} \frac{P_{n+1}(z)}{P_n(z)} \equiv
\lim_{n \rightarrow \infty} \frac{a_{n+1}p_{n+1}(z)}{p_n(z)}
\end{equation}
exists for all $z \in \C \setminus \R$, where $P_n$ are the monic orthogonal polynomials, $p_n$
are the orthonormal polynomials, and $a_n$ are the off-diagonal Jacobi parameters corresponding to $\mu$.
\end{definition}

As above, our strategy for dealing with ratio asymptotic measures
consists of identifying the appropriate class of right limits by
analyzing their invariants. We also want to emphasize the
relationship with weak asymptotic convergence. Let $H\left(
\{a_n,b_n\}_{n\in \Z} \right)$ be a whole-line Jacobi matrix. Let
$J_n^+=J\left( \{a_{j+n},b_{j+n}\}_{j=1}^\infty \right)$ be the
half-line Jacobi matrix one gets when restricting $H$ to $\ell^2(j >
n)$ with Dirichlet boundary conditions, and $J_n^-=J\left(
\{a_{n-j},b_{n+1-j}\}_{j=1}^\infty \right)$, the half-line Jacobi
matrix one gets when restricting $H$ to $\ell^2(j \leq n)$ with
Dirichlet boundary conditions. $J_n^+$ and $J_n^-$ have spectral
measures associated with them which we denote by $\mu_n^+$ and
$\mu_n^-$. Finally, for $z \in \C \setminus \R$ let $m_\pm(n;z)=\int
\frac{d\mu_n^\pm(x)}{x-z}$ be the corresponding Borel-Stieltjes
transforms. We are interested in $H$ for which these are constants
in $n$. The reason for this is the fact that if $H$ is a right limit
of $J$, then $-\frac{1}{m_-(0;z)}$ is a limit of
$\frac{P_{n+1}(z)}{P_n(z)}$ along an appropriate subsequence (see
e.g.\ \cite{Remling}---note that his $m_-$ is our $-1/m_-$). Thus,

\begin{theorem}\label{ratio-limits}
Let $H\left( \{a_n,b_n\}_{n\in \Z} \right)$ be a whole-line Jacobi matrix. Then the
following are equivalent: \\
$(i)$ $H$ belongs to Simon Class and its spectrum is a single interval. \\
$(ii)$ $a_{n}=a, \ b_n=b$ for some numbers, $a \geq 0$
and $b \in \R$ and all $n\in\Z$.  \\
$(iii)$ $m_-(n;z)=m_-(n+1;z)$ for all $z \in \C \setminus \R$, $n\in\Z$. \\
$(iv)$ $m_+(n;z)=m_+(n+1;z)$ for all $z \in \C \setminus \R$, $n\in\Z$. \\
$(v)$ $m_-(n;z)=m_-(n+1;z)$ for some $z \in \C \setminus \R$ and all $n\in\Z$. \\
$(vi)$ $m_+(n;z)=m_+(n+1;z)$ for some $z \in \C \setminus \R$ and
all $n\in\Z$.
\end{theorem}

\begin{proof}
$(i) \Leftrightarrow (ii)$ follows from the theory of periodic
Jacobi matrices (see \cite[Sect.\ 7.4]{Teschl}). $(ii) \Rightarrow
(iii) \Rightarrow (v)$ and $(ii) \Rightarrow (iv) \Rightarrow (vi)$
are clear by periodicity. Thus we are left with showing $(v)
\Rightarrow (ii)$ and $(vi) \Rightarrow (ii)$. Writing down the
continued fraction expansion for $m_-(n;z)$:
\[
-\frac{1}{m_-(n+1;z)}=z-b_{n+1}+a_n^2m_-(n;z), \quad n\in\Z,
\]
one sees that $(v)$ implies $m_-(n;z)$ satisfies a quadratic
equation. $a_n$ and $b_{n+1}$ are then determined from this equation
by taking imaginary and real parts, and so we get $(v) \Rightarrow
(ii)$ (see the proof of Theorem 2.2 in \cite{Simon-Weak} for
details). The same can be done for $m_+(n;z)$ to get $(vi)
\Rightarrow (ii)$.
\end{proof}

By the above discussion and Theorem \ref{ratio-limits}, it follows that $\mu$ is
ratio asymptotic if and only if its Jacobi matrix has a unique right limit in
Simon Class with constant off-diagonal elements. Moreover, $(v)$ in Theorem
\ref{ratio-limits} implies that it is enough to require ratio asymptotics at a
single $z \in \C \setminus \R$. This is precisely the content of Theorem 1 in
\cite{Simon-Weak}. We shall show below that the same strategy can be applied in
the OPUC case in order to get a strengthening of corresponding results by
Khrushchev.

\subsection{The Khrushchev Class}

We now turn to the discussion of the analogous theory for half-line
CMV matrices. Namely, we study CMV matrices with the property that
$d\mu_n(\theta)=|\varphi_{n}(e^{i\theta})|^2d\mu(\theta)$ has a weak
limit as $n \rightarrow \infty$. Again, as is clear from the
discussion in the introduction, all these right limits belong to
Khrushchev Class (recall Definition \ref{Khrushchev Class}) and so
the analysis is mainly the analysis of properties of that class.
Since nontrivial CMV matrices can have many powers with zero
diagonal, the computations are substantially more complicated. Here
is the analog of Theorem \ref{Simon Class Characterization}:

\begin{theorem}\label{Khrushchev Class Characterization}
Let $\EC$ be a whole-line CMV matrix and $k\in\N\cup\{\infty\}$. Then the
following are equivalent: \\
$(i)$ $\EC$ belongs to Khrushchev Class with $[\EC^\ell]_{n,n}=0$
for $\ell=1,\dots,k-1$ and all $n\in\Z$, and in the case $k<\infty$,
$[\EC^k]_{n,n}=c$ for some $c\in\ol\D\setminus\{0\}$ and all
$n\in\Z$.
\\
$(ii)$ For $\ell=1,\dots,k-1$,
\begin{align*}
\int_0^{2\pi}e^{i\ell\te}\,d\mu_n(\te)=0, \quad n\in\Z,
\end{align*}
and if $k<\infty$ then additionally, for some
$c\in\ol\D\setminus\{0\}$,
\begin{align*}
\int_0^{2\pi}e^{ik\te}\,d\mu_n(\te)=c, \quad n\in\Z.
\end{align*}
$(iii)$ There exist $n_0\in\N$, $a,b\in(0,1]$, and $t\in[0,2\pi)$
such that in the case $k<\infty$,
\begin{align*}
&|\al_{n_0+2nk}|=a, \quad |\al_{n_0+(2n+1)k}|=b, \\
&\al_{n_0+nk+j}=0,  \quad \arg(\ol\al_{n_0+(n+1)k}\al_{n_0+nk})=t, \quad
n\in\Z, \; j=1,\dots,k-1,
\end{align*}
and in the case $k=\infty$,
\begin{align*}
&\al_{j}=0, \quad j\in\Z\setminus\{n_0\}.
\end{align*}
\end{theorem}
\begin{remark}
In particular, this shows that the constancy of the first $k$
moments, where the $k$-th moment is the first nonzero one, implies
that $\EC$ belongs to Khrushchev Class. Note, however, that the
value of the $k$-th moment does not determine the element of the
class itself (again, not even up to translation; see
Theorem~\ref{Khrushchev-moment} below). Thus, it makes sense to
define $\cK(c,k)$, for $k < \infty$, to be the set of all matrices
in the Khrushchev Class with $[\EC^\ell]_{n,n}=0$ for all $n\in\Z$,
$\ell=1,\dots,k-1$, and $[\EC^k]_{n,n}=c\neq0$ for all $n\in\Z$. In
the case $k = \infty$, let $\cK(\infty)$ be the set of all matrices
with $[\EC^\ell]_{n,n}=0$ for all $n\in\Z$, $\ell \geq 1$. We note
that every CMV matrix $\EC$ from the Khrushchev Class belongs to one
of $\cK(c,k)$, $c\in\ol\D\setminus\{0\}$, $k\in\N$, or to
$\cK(\infty)$.
\end{remark}

\begin{proof}
$(i) \Rightarrow (ii)$: Follows from
\begin{align} \lb{moments}
\int_0^{2\pi}e^{i\ell\te}\,d\mu_n(\te) = [\EC^\ell]_{n,n} \,\text{ for all }\,
\ell\in\N,\; n\in\Z.
\end{align}

$(ii) \Rightarrow (iii)$: First, observe that $(iii)$ is equivalent to the
following L\'{o}pez-type condition: there exists $n_0\in\Z$ such that for all
$n\in\Z$, $\ell=1,\dots,k$, $j=0,\dots,\ell-1$,
\begin{align} \lb{Lopez-condition}
\ol\al_{n_0+n\ell+j}\al_{n_0+(n-1)\ell+j} =
\begin{cases}
abe^{it} & j=0,\; \ell=k,\\
0 & \text{otherwise}.
\end{cases}
\end{align}
We will show that \eqref{Lopez-condition} holds with $abe^{it}=-c$ by verifying
inductively with respect to $\ell$ that
\begin{align} \lb{Lopez-moments}
-\ol\al_{n_0+n\ell+j}\al_{n_0+(n-1)\ell+j} =
\begin{cases}
\int_{0}^{2\pi}e^{i\ell\te}\,d\mu_{n_0+n\ell} & j=0,\\
0 & \text{otherwise}
\end{cases}
\end{align}
for some $n_0\in\Z$ and all $n\in\Z$, $\ell=1,\dots,k$,
$j=0,\dots,\ell-1$.

The case $\ell=1$ trivially follows from \eqref{whole-line-cmv}
since the first moment of $\mu_n$ is exactly the diagonal element
$\EC_{n,n}$ for all $n\in\Z$.

Now suppose \eqref{Lopez-moments} holds for $\ell=1,\dots,p-1$ for some $p \leq
k$. In view of \eqref{moments}, we want to compute
$[\EC^p]_{n,n}=[(\cL\cM)^p]_{n,n}$. To do this, it turns out to be useful to
separate the diagonal and off-diagonal elements of $\cL$ and $\cM$ and identify
the contributions to the product. Thus, let the diagonal matrices $X_{-1}=
\textrm{diag} \cL$ and $X_{1}=\textrm{diag} \cM$ be the diagonals of $\cL$  and
$\cM$ respectively. Furthermore, define $Y_{-1}$ and $Y_1$ through
$\cL=X_{-1}+Y_{-1},$ $\cM=X_1+Y_1$. Expressed in this notation, our objective
is to compute the diagonal elements of \begin{equation}
\label{diagonal-factorization} \EC^p=\left(X_{-1}+Y_{-1}
\right)\left(X_{(-1)^2}+Y_{(-1)^2} \right)\cdots
\left(X_{(-1)^{2p}}+Y_{(-1)^{2p}} \right).
\end{equation}

First, it is a direct computation to verify that for any two $s, r \in \N$,
\begin{equation} \label{Ys}\textrm{diag} Y_{(-1)^s}Y_{(-1)^{s+1}} \cdots
Y_{(-1)^{s+r}}=0. \end{equation} Now, using $(ii)$  and the induction
hypothesis (\eqref{Lopez-moments} for $\ell \leq p-1$) one verifies that
\begin{align}
&[Y_{(-1)^{j-s}} Y_{(-1)^{j-s+1}} \cdots Y_{(-1)^{j-1}} X_{(-1)^j}
Y_{(-1)^{j+1}} \cdots Y_{(-1)^{j+s-1}} Y_{(-1)^{j+s}}]_{n,n} \no \\
&\quad =
\begin{cases}
\ol\al_{n+s}\rho^2_{n}\cdots\rho^2_{n+s-1} & \text{$n+s+j$ is odd},\\
-\al_{n-s-1}\rho^2_{n-s}\cdots\rho^2_{n-1} & \text{$n+s+j$ is even}
\end{cases} \no \\
&\quad =
\begin{cases}
\ol\al_{n+s} & \text{$n+s+j$ is odd},\\
-\al_{n-s-1} & \text{$n+s+j$ is even},
\end{cases} \quad n,j\in\Z, \; s=0,\dots,p-1, \lb{YXYdiag}
\end{align}
and
\begin{align}
&[Y_{(-1)^{j-s}} Y_{(-1)^{j-s+1}} \cdots Y_{(-1)^{j-1}} X_{(-1)^j}
Y_{(-1)^{j+1}} \cdots Y_{(-1)^{j+s-1}} Y_{(-1)^{j+s}}]_{n,m}=0 \lb{YXYoffdiag}
\end{align}
whenever $n \neq m$.

This identity combined with the induction hypothesis, $(ii)$, \eqref{moments},
\eqref{diagonal-factorization}, and \eqref{Ys} implies that (for notational
simplicity we let $\wti{Y}_j=Y_{(-1)^j}$, $\wti{X}_j=X_{(-1)^j}$)
\begin{align*}
&\int_0^{2\pi}e^{ip\te}\,d\mu_n(\te) = [\EC^p]_{n,n} \\
&\quad = \sum_{\ell=1}^p [\wti{Y}_1 \cdots \wti{Y}_{\ell-1}
\wti{X}_\ell \wti{Y}_{\ell+1} \cdots
\wti{Y}_{p+\ell-1} \wti{X}_{p+\ell} \wti{Y}_{p+\ell+1} \cdots \wti{Y}_{2p}]_{n,n}\\
&\quad = -\sum_{\ell=1}^p
\begin{cases}
\ol\al_{n+p-\ell}\al_{n-\ell} & \text{$n$ is odd},\\
\ol\al_{n+\ell-1}\al_{n-p+\ell-1} & \text{$n$ is even}
\end{cases} \\
&\quad = -\sum_{\ell=1}^p \ol\al_{n+p-\ell}\al_{n-\ell}, \quad
n\in\Z.
\end{align*}
The idea behind the computation is that all summands containing no $X$, a
single $X$, or two $X$'s that are a distance greater than or less than $p$
apart do not contribute to the diagonal. This follows from the induction
hypothesis and \eqref{Ys}--\eqref{YXYoffdiag}.

Now, observe that the sum in the above equality may have at most one
nonzero term. Indeed, if there are no nonzero terms we are done
(this may happen only if $p<k$), otherwise let $n_0\in\Z$ be such
that $\ol\al_{n_0}\al_{n_0-p}\neq0$. Then combining the induction
hypothesis \eqref{Lopez-moments} with $(ii)$ yields,
\[
\ol\al_{n+\ell}\al_{n} = 0, \quad n\in\Z, \; \ell=1,\dots,p-1
\]
which together with $\ol\al_{n_0}\al_{n_0-p}\neq0$ implies
\[
\al_{n_0+np+\ell}=0, \quad n\in\Z, \; \ell=1,\dots,p-1.
\]
Hence, when $p<k$,
\[
0=\int_0^{2\pi}e^{ip\te}\,d\mu_{n_0+np}(\te) =
-\ol\al_{n_0+np}\al_{n_0+(n-1)p}, \quad n\in\Z,
\]
and carrying the induction up to $k$,
\[
c = \int_0^{2\pi}e^{ik\te}\,d\mu_{n_0+nk}(\te) =
-\ol\al_{n_0+nk}\al_{n_0+(n-1)k}, \quad n\in\Z.
\]
Thus, \eqref{Lopez-moments} holds for $\ell=k$, and hence $(iii)$ follows from
\eqref{Lopez-condition} with $abe^{it}=-c$.

$(iii) \Rightarrow (i)$: First, note that if $k=\infty$ then there is at most
one nonzero Verblunsky coefficient and hence by Example~\ref{example}, $\EC$ is
in the Khrushchev Class with $d\mu_n(\te)=\f{d\te}{2\pi}$ for all $n\in\Z$, and
hence it follows from \eqref{moments} that $[\EC^\ell]_{n,n}=0$ for all
$\ell\in\N$ and $n\in\Z$.

Next suppose $k<\infty$. It follows from $(iii)$ that there are
$t_0,t\in[0,2\pi)$ such that
\[
\al_{n_0+n}=|\al_{n_0+n}|e^{i(t_0+tn)} \,\text{ for all }\, n\in\Z.
\]
Then, using the Schur algorithm, one finds the following relations
between the functions $f_\pm$ associated with
$\al=\{\al_n\}_{n\in\Z}$ and $|\al|=\{|\al_n|\}_{n\in\Z}$,
respectively,
\begin{align*}
&f_+(z,n_0+n;\al)=e^{i(t_0+tn)}f_+(e^{-it}z,n_0+n;|\al|),\\
&f_-(z,n_0+n;\al)=e^{-i(t_0+t(n-1))}f_-(e^{-it}z,n_0+n;|\al|), \quad n\in\Z, \; z\in\D.
\end{align*}
Hence by \eqref{diag-schur} the diagonal Schur functions associated with $\al$
and $|\al|$ are related by
\begin{align} \lb{abs-al}
f(z,n_0+n;\al) = e^{it}f(e^{-it}z,n_0+n;|\al|), \quad n\in\Z, \; z\in\D.
\end{align}

Now, the conditions in $(iii)$ imply,
\begin{align*}
&f_+(z,n_0+nk+j;|\al|) = z^{k-j}f_+(z,n_0+(n+1)k;|\al|),\\
&f_-(z,n_0+nk+j;|\al|) = z^{j-1}f_-(z,n_0+nk+1;|\al|),\\
&f_+(z,n_0+nk;|\al|) = f_+(z,n_0+(n \mod 2)k;|\al|),\\
&f_-(z,n_0+nk+1;|\al|) = -f_+(z,n_0+nk;|\al|), \quad n\in\Z, \;
j=1,\dots,k, \; z\in\D.
\end{align*}

These identities together with \eqref{diag-schur} and \eqref{abs-al}
yield
\begin{align*}
&f(z,n_0+nk+j;\al)\\
&\quad = e^{-it(k-2)}z^{k-1}f_+(e^{-it}z,n_0+(n+1)k;|\al|)f_-(e^{-it}z,n_0+nk+1;|\al|)\\
&\quad = -e^{-it(k-2)}z^{k-1}f_+(e^{-it}z,n_0+(n+1)k;|\al|)f_+(e^{-it}z,n_0+nk;|\al|)\\
&\quad =
-e^{-it(k-2)}z^{k-1}f_+(e^{-it}z,n_0+k;|\al|)f_+(e^{-it}z,n_0;|\al|)
\end{align*}
for all $n\in\Z$, $j=1,\dots,k$, $z\in\D$. Hence
$f(\cdot,n;\al)=f(\cdot,m;\al)$ which is equivalent to $\mu_m=\mu_n$ for all
$m,n\in\Z$.

The presence of the factor $z^{k-1}$ implies that the first $k-1$ moments of
$\mu_n$, $n\in\Z$, are zero. This follows from the relationship
\eqref{carat-schur} between Schur functions and Carath\'eodory functions and
the fact that Taylor coefficients of $F$ are twice the complex conjugates of
the moments of $\mu$. Moreover, since $z^{-k+1}f(z,n_0+nk+j;\al)$ is nonzero at
the origin, the $k$-th moment of $\mu_n$, $n\in\Z$, is nonzero, and hence one
gets $(i)$ from \eqref{moments}.
\end{proof}

\begin{corollary}\label{generalized-Khrushchev}
Let $\cmv$ be a half-line CMV matrix and let $\mu$ be its spectral
measure. For $n \geq 0$, let $d\mu_n(\te) =
|\varphi_{n}(e^{i\te})|^2\,d\mu(\te)$ be the spectral measure of $C$
and $\delta_n$. If for some $c\in\ol\D\setminus\{0\}$ and
$k\in\N\cup\{\infty\}$,
\begin{equation} \label{k-moments}
\lim_{n \rightarrow \infty} \int_0^{2\pi} e^{i\ell\te}\, d\mu_n(\te) =
\begin{cases}
0 & \ell=1,\dots,k-1,\\
c & \ell=k,\; k<\infty,
\end{cases}
\end{equation}
then all right limits of \,$\cmv$ are in $\cK(c,k)$ if $k<\infty$ or
in $\cK(\infty)$ if \,$k=\infty$.
\end{corollary}

The analogy with Corollary \ref{generalized-Simon} should be clear.
Corollary~\ref{generalized-Khrushchev} is a variant of Theorem E in
\cite{Khrushchev} with weaker assumptions and weaker conclusions.
Our proof is new and based on a completely different approach. We
also note that much the same as in the Jacobi case, convergence of
the first $k$-moments does not imply weak convergence, but by
Corollary \ref{generalized-Khrushchev} it does imply the same weak
form of weak convergence: convergence holds along any subsequence on
which $\cmv$ has a right limit.

A notable difference between the OPUC and OPRL cases is the fact
that multiplication of the Verblunsky coefficients by a constant
phase does not change the spectral measures. Thus, even when $\mu_n$
converges weakly, it is not possible to deduce uniqueness of a right
limit (even up to a shift). Note that the phase ambiguity is
equivalent to a choice of $t_0$ in the proof of Theorem
\ref{Khrushchev Class Characterization} and there is no way to
determine this $t_0$ from information on $\mu_n$ alone. In the case
$k=\infty$ even $|\alpha_{n_0}|$ cannot be determined from the
information on the measure and so the indeterminacy is, in a sense,
even more severe.

That said, as in the Jacobi case, it is clear that when $k < \infty$
any condition forcing weak convergence of $\mu_n$ (in addition to
those in Corollary \ref{generalized-Khrushchev}) is equivalent to a
condition that distinguishes an element of $\cK(c,k)$ (up to a shift
and multiplication by an arbitrary phase). In particular, a somewhat
tedious computation (along the lines of the argument in $(ii)
\Rightarrow (iii)$ above) shows that the following result holds:

\begin{theorem}\label{Khrushchev-moment}
Suppose that $k<\infty$, \eqref{k-moments} holds, and $\lim_{n \rightarrow \infty}
\int_0^{2\pi} e^{2ik\te}\,d\mu_n(\te)$ exists.  Then $d\mu_n$ converge weakly
and $\cmv$ has a unique right limit $($up to a shift and multiplication by a
constant phase$)$ in $\cK(c,k)$.
\end{theorem}

\begin{remark}
We note that on the level of Verblunsky coefficients,
Theorems~\ref{Khrushchev Class Characterization} and
\ref{Khrushchev-moment} imply for $k<\infty$,
\begin{align}
\begin{split}\lb{lopez}
&\lim_{n\to\infty}|\al_{n_0+2nk+j}|=
\begin{cases}a&j=0,\\b&j=k,\\0&j\in\{1,\dots,k-1,k+1,\dots,2k-1\},\end{cases}
\\
&\lim_{n\to\infty}\ol\al_{n_0+(n+1)k}\al_{n_0+nk}=-c, \,\text{ for some }\,
n_0\in\Z \text{ and } ab=|c|,
\end{split}
\end{align}
and similarly, Theorem~\ref{Khrushchev Class Characterization} and
Corollary~\ref{generalized-Khrushchev} imply for $k=\infty$,
\begin{align}\lb{mate-nevai}
\lim_{n\to\infty}|\al_{n_0+n}\al_{n}|=0 \,\text{ for any }\, n_0\in\Z.
\end{align}
This extends \cite[Thm.\ E]{Khrushchev}, where the stronger
condition of weak convergence for the measures $d\mu_n$ is assumed.
\end{remark}

Next, we use right limits to study ratio asymptotics. It is
convenient to introduce $\wti\cK(c,1)$ as the subclass of $\cK(c,1)$
consisting of CMV matrices with Verblunsky coefficients of constant
absolute value.

\begin{definition}
Let $\mu$ be a probability measure on the unit circle. We say $\mu$
is \emph{ratio asymptotic} if
\begin{equation} \nonumber
\lim_{n\to\infty}\f{\Phi^*_{n+1}(z)}{\Phi^*_n(z)}
\end{equation}
exists for all $z \in \D$, where, as usual, $\Phi_n(z)$ is the
degree $n$ monic orthogonal polynomial associated to $\mu$.

In particular, we say \emph{ratio asymptotics} holds at $z \in \D$
with limit $G(z)$ if
\begin{equation} \lb{ra}
\lim_{n\to\infty}\f{\Phi^*_{n+1}(z)}{\Phi^*_n(z)}=G(z).
\end{equation}
\end{definition}

\begin{theorem}
Let $\Phi_n$ be the monic orthogonal polynomials associated with a
half-line CMV matrix $\cmv$. If either all right limits of \,$\cmv$
are in $\cK(\infty)$ or \,$\cmv$ has a unique right limit $($up to a
multiplication by a constant phase$)$ in $\wti\cK(c,1)$, then $\mu$
is ratio asymptotic.

Conversely, if ratio asymptotics holds at some point
$z_0\in\D\setminus\{0\}$ with limit $G(z_0)=1$, then all right
limits of $\cmv$ are in $\cK(\infty)$.  If ratio asymptotics holds
at two points $z_1,z_2\in\D\setminus\{0\}$ and the limit is not $1$
at either point, then $\cmv$ has a unique right limit $($up to
multiplication by a constant phase$)$ in $\wti\cK(c,1)$ for some
$c\in\ol\D\setminus\{0\}$.
\end{theorem}

\begin{proof}
First, observe that it follows from the Szeg\H{o} recursion \eqref{SzRec} that
for all $n\in\Z_+$ and $z\in\D$,
\begin{align}\lb{phi-f}
1-\f{\Phi^*_{n+1}(z)}{\Phi^*_n(z)} = z\al_n\f{\Phi_n(z)}{\Phi^*_n(z)} = z\al_n
f(z;-\ol\al_{n-1},-\ol\al_{n-1},\dots,-\ol\al_0,1).
\end{align}
We refer to \cite[Prop.\ 9.2.3]{OPUC2} for the details on the second
equality. Abbreviating by
$f_n(z)=f(z;-\ol\al_{n-1},-\ol\al_{n-1},\dots,-\ol\al_0,1)$, we see
that ratio asymptotics \eqref{ra} at $z\in\D\setminus\{0\}$ is
equivalent to $\lim_{n\to\infty}\al_nf_n(z)=g(z)\equiv(1-G(z))/z$.

Let $\EC$ be a right limit of $\cmv$ and $\be_n$, $f_\pm(\cdot,n)$, $n\in\Z$,
be the Verblunsky coefficients and Schur functions associated with $\EC$. Then,
$\be_nf_-(z,n)=\lim_{j\to\infty}\al_{n+n_j}f_{n+n_j}(z)$ for all $n\in\Z$,
$z\in\D$, and some sequence $\{n_j\}_{j\in\N}$.

By Theorem~\ref{Khrushchev Class Characterization}, if
$\EC\in\cK(\infty)$ then at most one $\be_n$ is nonzero, and hence
$\be_0f_-(z,0)=0$ for all $z\in\D$. If $\EC\in\wti\cK(c,1)$ then
$|\be_n|=\sqrt{|c|}$ and $\ol\be_{n+1}\be_n=-c$ for all $n\in\Z$.
Then the Schur algorithm implies that $\be_0f_-(z,0)$ is a function
that depends only on the value of $c$. Since in both cases
$\be_0f_-(z,0)$ is independent of the sequence $n_j$, it follows
that $\lim_{n\to\infty}\al_nf_n(z)=\be_0f_-(z,0)$ for all $z\in\D$.
Thus, by \eqref{phi-f}, ratio asymptotics holds for all $z\in\D$.

Conversely, by \eqref{phi-f}, ratio asymptotics at
$z\in\D\setminus\{0\}$ implies $\be_nf_-(z,n)=g(z)$ for all
$n\in\Z$. By the Schur algorithm we have
\begin{align}\lb{sa1}
f_-(z,n+1)[1-zg(z)]=zf_-(z,n)-\ol\be_n \,\text{ for all }\, n\in\Z.
\end{align}

If \eqref{ra} holds at $z_0\neq0$ and the limit is $1$, then by
\eqref{phi-f} $g(z_0)=0$. Thus, by \eqref{sa1} there is at most one
nonzero $\be_n$ since $\be_{n_0}\neq0$ implies inductively that
$f_-(z_0,n)\neq0$ and hence $\be_n=0$ (since $g(z_0)=0$) for all
$n>n_0$. By Theorem~\ref{Khrushchev Class Characterization},
$\EC\in\cK(\infty)$.

Finally consider the case where ratio asymptotics holds at two
different points $z_1, z_2 \in \D\setminus \{0\}$ and the limit is
not $1$ at either point.  Then by \eqref{phi-f}, $g(z_1)\neq0$ and
$g(z_2)\neq0$ and hence $\be_n\neq0$ for all $n\in\Z$. We also see
that $z_1g(z_1)\neq z_2g(z_2)$ since otherwise it follows from
\eqref{sa1} that $z_1=z_2$. Thus, multiplying \eqref{sa1} by
$\be_{n+1}/(zg(z))$, substituting $z=z_j$, $j=1,2$, and subtracting
the results then yields
\begin{align}\lb{sa2}
\be_{n+1}\ol\be_n =
\f{\f{1}{z_1}-g(z_1)-\f{1}{z_2}+g(z_2)}{\f{1}{z_1g(z_1)}-\f{1}{z_1g(z_1)}}
\,\text{ for all }\, n\in\Z.
\end{align}
Similarly, multiplying $\eqref{sa1}$ by $|\be_n|^2\be_{n+1}$ and evaluating at
$z=z_1$ one obtains,
\begin{align}\lb{sa3}
|\be_n|^2=\f{z_1g(z_1)\be_{n+1}\ol\be_n}{g(z_1)(1-z_1g(z_1))+\be_{n+1}\ol\be_n}
\,\text{ for all }\, n\in\Z.
\end{align}
By \eqref{sa2} the right-hand side of \eqref{sa3} is $n$-independent, and hence
$|\be_n|$ as well as $\be_{n+1}\ol\be_n$ are $n$-independent constants uniquely
determined by the ratio asymptotics \eqref{ra} at $z_1$ and $z_2$. Thus,
$\EC\in\wti\cK(c,1)$ with $c=-\ol\be_{n+1}\be_n\neq0$. Since $c$ is determined
by the ratio asymptotics at $z_1$ and $z_2$, all right limits are the same up
to multiplication by a constant phase.
\end{proof}

\begin{remark}
This theorem extends an earlier result of Khrushchev \cite[Thm.\
A]{Khrushchev}.
\end{remark}

We conclude with the
\begin{proof}[Proof of Theorem \ref{Simon and Khrushchev Class spectrum}]
Let $H$ belong to the Simon Class and let $a,b,c$ be as in $(iii)$
of Theorem~\ref{Simon Class Characterization}. If $a,c>0$ then $H$
is a periodic whole-line Jacobi matrix which is well known to be
reflectionless on its spectrum (\cite{Teschl}). If $a=0$ or $c=0$
then $H$ is a direct sum of identical $2 \times 2$ (or $1 \times 1$)
self-adjoint matrices and so has pure point spectrum of infinite
multiplicity supported on at most two points.

For $\EC$ in the Khrushchev Class with $k< \infty$ the same analysis goes
through: as long as $|a|,|b|<1$ we get a reflectionless operator. If one of
them or both are unimodular then it is easy to see that $\EC$ is a direct sum
of $2 \times 2$ (or $1 \times 1$) matrices. If $k=\infty$ then $\EC$ is either
reflectionless or belongs to the class of matrices from Example~\ref{example}.
\end{proof}

%
%
%
%


\end{document}